\documentclass[a4paper, 11 pt, twoside]{amsart}
\usepackage[utf8]{inputenc}
\usepackage[english]{babel}
\usepackage[all]{xy}
\usepackage{url}
\usepackage{tikz}
\usepackage{enumitem}
\usepackage{xspace}
\newlist{prooflist}{description}{1}
\setlist[prooflist]{font=\normalfont \itshape, labelindent = \parindent, leftmargin = 0pt}

\RequirePackage[dvipsnames, pdftex]{xcolor}
\RequirePackage{tikz, tikz-cd, tcolorbox}
\usepackage{soul}
\setstcolor{blue}
\usepackage[unicode,bookmarks, pdftex]{hyperref}
\hypersetup{colorlinks=true,citecolor=NavyBlue,linkcolor=NavyBlue,urlcolor=Orange, pdfpagemode=UseNone, breaklinks=true}

\usepackage{amsmath}

\usepackage{amssymb}

\usepackage{amsthm}
\usepackage{mathrsfs}

\usepackage{mathtools}

\usepackage{amsfonts}

\usepackage{todonotes}

\usepackage[left = 3.5cm, right = 3.5cm, headsep = 6mm,
footskip = 10mm, top = 35mm, bottom = 35mm, footnotesep=5mm, headheight =
2cm]{geometry}

\numberwithin{equation}{section}

\newcommand{\del}{\partial}
\newcommand{\dbar}{\overline{\partial}}
\newcommand{\dbarb}{\overline{\partial}_b}
\newcommand{\SU}{\mathrm{SU}}

\newcommand{\Wedge}{\textstyle\bigwedge}
\DeclareMathOperator{\Hom}{Hom}

\newcommand{\dint}{\ \!\mathrm{d}}
\newcommand{\dext}{\mathrm{d}}

\newcommand{\Z}{\mathbb{Z}}
\newcommand{\N}{\mathbb{N}}
\newcommand{\C}{\mathbb{C}}
\newcommand{\R}{\mathbb{R}}
\newcommand{\Q}{\mathbb{Q}}

\newcommand{\T}{\mathbb{T}}

\newcommand{\Ad}{\operatorname{Ad}}

\newcommand{\Tr}{\operatorname{Tr}}
\newcommand{\Ann}{\operatorname{Ann}}
\newcommand{\HS}{\mathrm{HS}}

\newcommand{\End}{\operatorname{End}}

\newcommand{\CRO}{\operatorname{CR0}}
\newcommand{\CRI}{\operatorname{CR1}}

\newcommand{\psum}{\sideset{}{'}\sum}

\newcommand{\Kon}[1]{}

\newcommand{\Max}[1]{}

\newcommand{\DC}{{\bf (DC)}\xspace}

\newtheoremstyle{Lehn-Bemerkung}
  {}
  {}
  {}
  {}
  {\itshape}
  {$\;$\textmd{---}}
  { }
  {}

\newtheorem{dfn}[equation]{Definition}
{\theoremstyle{Lehn-Bemerkung}
\newtheorem{rmk}[equation]{Remark}
}
{\theoremstyle{Lehn-Bemerkung}
\newtheorem{exa}[equation]{Example}
}
\newtheorem{lem}[equation]{Lemma}
\newtheorem{prp}[equation]{Proposition}
\newtheorem{thm}[equation]{Theorem}
\newtheorem{cor}[equation]{Corollary}

\usetikzlibrary{matrix}

\DeclareUnicodeCharacter{2202}{\ensuremath{\partial{}}}

\DeclareUnicodeCharacter{0305}{\ensuremath{\overline{\phantom{L}}}\!\!\!\!}

\begin{document}

\title{Cohomology of CR structures on compact Lie groups}

\author[H. Jacobowitz]{Howard Jacobowitz}
\address{Rutgers University}
\email{jacobowi@camden.rutgers.edu}

\author[M. R. Jahnke]{Max Reinhold Jahnke }
\thanks{The second and fourth authors were funded by the DFG (grant JA 3453/2-1 and RO 3734/4-1, respectively) during this work.}
\address{Universität zu Köln}
\email{max.jahnke@uni-koeln.de}

\author[V. Novelli]{Vinícius Novelli}
\thanks{The third author was funded by Sao Paulo Research Foundation (grant 2023/17607-7) and FWF (grant 10.55776/PAT3705925) during this work.}
\address{University of Vienna}
\email{vinicius.novelli.da.silva@univie.ac.at}

\author[K. Wehler]{Konstantin Wehler}
\address{Philipps-Universität Marburg}
\email{wehlerk@staff.uni-marburg.de}

\begin{abstract}
    We show that, under a division condition, the tangential Cauchy--Riemann cohomology of a compact Lie group with a left-invariant CR structure can be computed on a suitable maximal torus. As a consequence, we conclude that the tangential Cauchy--Riemann cohomology is finite-dimensional. We also show that, for a class of CR structures, this division condition is necessary for the total cohomology to be finite-dimensional. The proof combines Fourier analysis on compact Lie groups, highest-weight representations and Lie algebra cohomology. This not only generalizes but provides completely new proofs for the analogous result due to Pittie and for its extensions to Levi-flat CR structures, obtained by Jacobowitz and Jahnke.
\end{abstract}

\maketitle
\setcounter{tocdepth}{1}

\section{Introduction}
\label{sec:introduction}

To a given compact, smooth CR manifold, one can associate a complex of first-order differential operators, known as the \textit{tangential Cauchy--Riemann} complex, which reduces to the usual Dolbeault complex when the manifold is complex. It possesses many interesting analytic properties, both locally and globally \cite{lewy_example_1957}. Since it is not an elliptic complex in general, the corresponding cohomology spaces can fail to have finite dimension. This property is implied by certain sign conditions on the eigenvalues of the Levi form \cite{folland_kohn}, but it is, in general, a wide open question to determine the properties which characterize finite-dimensionality of these cohomology spaces. In this work, we consider left-invariant CR structures on compact Lie groups and provide fairly complete conditions that characterize finite-dimensionality.

Let $G$ be a connected compact Lie group with Lie algebra $\mathfrak g$, and let $\mathfrak g_\C$ be its complexification. A left-invariant CR structure on $G$ is a left-invariant involutive subbundle $\mathcal{V}\subset T_\C G$ satisfying $\mathcal{V} \cap \overline{\mathcal {V}} = \{0\}$. Therefore, the CR structure $\mathcal{V}$ is uniquely determined by a subalgebra $\mathfrak v \subset \mathfrak g_\C$.

Under the hypothesis $\mathfrak v \oplus \overline{\mathfrak v} = \mathfrak g_\C$, the subalgebra $\mathfrak v$ defines a left-invariant complex structure on $G$, which is always known to exist if $G$ is even-dimensional \cite{wang_closed_1954}. In \cite{pittie_dolbeault-cohomology_1988}, Pittie studied the Dolbeault complex of compact Lie groups with a left-invariant complex structure and proved that the cohomology is computed by the restriction of this complex to a suitable maximal torus in $G$.

If $G$ is odd-dimensional, then it follows from the Charbonnel--Khalgui classification theorem \cite{charbonnel_classification_2004} that there always exists a subalgebra $\mathfrak v \subset \mathfrak g_\C$ defining a left-invariant CR structure of hypersurface type, that is, $\mathfrak v \oplus \overline{\mathfrak v}$ has codimension one in $\mathfrak g_\C$. Moreover, there is a maximal torus $\T \subset G$, whose Lie algebra we denote by $\mathfrak t$, and a Lie algebra $\mathfrak m \subset \mathfrak t_\C \cap \mathfrak v$, known as the toral part of $\mathfrak v$, defining a bi-invariant CR structure on $\T$.

The tangential Cauchy--Riemann cohomology $H^{q}(G, \mathfrak v)$ of a left-invariant CR structure $\mathfrak v \subset \mathfrak g_\C$ on $G$ is defined by the tangential Cauchy--Riemann complex $\left( \mathscr C^\infty(G, \underline \Lambda^\bullet), \dbarb \right)$. For the precise construction of this complex, see Section \ref{sec:involutive_structures} or \cite{boggess_cr_1991}. 
 
 It was proved in \cite{jacobowitz_levi-flat_2023} that, under a certain division condition on $\mathfrak m$ denoted by \DC (Section \ref{sec:division_condition}), 
    the tangential Cauchy--Riemann complex of the maximal torus $\mathbb T$ with the induced CR structure $\mathfrak m$ forms a subcomplex of the above complex inducing an inclusion
    \begin{equation}\label{eq: inclusion}
     H^{q}(\mathbb T, \mathfrak{m}) \hookrightarrow H^{q}(G, \mathfrak{v}),
    \end{equation}
    and the cohomology $H^{q}(\mathbb T, \mathfrak{m})$ is finite-dimensional. In fact, the cohomology space $H^{q}(\mathbb T, \mathfrak{m})$ is finite-dimensional if and only if $\mathfrak m$ satisfies the division condition \DC. 

 In the case where $G$ is non-abelian, the only CR structures where the map \eqref{eq: inclusion} is known to be an isomorphism are Levi-flat CR structures \cite{jacobowitz_levi-flat_2023}, which are precisely the CR structures determining a foliation on $G$ by complex manifolds.
Our main result shows that the condition \DC yields the desired isomorphism for CR structures on all compact Lie groups.

 \begin{thm} \label{thm:main}
        Let $G$ be a connected, compact Lie group with a left-invariant CR structure $\mathfrak{v}$ of hypersurface type. If the toral part $\mathfrak{m}$ of $\mathfrak{v}$ satisfies \DC, then the inclusion
        \[
             H^{q}(\mathbb T, \mathfrak{m}) \hookrightarrow \ H^{q}(G, \mathfrak{v})
        \]
        is an isomorphism. In particular, $H^{q}(G, \mathfrak{v})$ is finite-dimensional and $H^{q}(G, \mathfrak{v})$ vanishes for $q > \dim \mathfrak m$. 
    \end{thm}

    If $G$ is even-dimensional, then the toral part of any left-invariant complex structure, which is the analogue of a CR structure of hypersurface type in this case, satisfies \DC and the proof of Theorem \ref{thm:main} can be easily adapted to recover Pittie’s result \cite{pittie_dolbeault-cohomology_1988}. Furthermore, the operator $\dbarb$ also maps left-invariant sections of $\underline \Lambda^{q}$ to left-invariant sections of $\underline \Lambda^{q+1}$.
    Thus, we have a subcomplex
    \begin{equation}
        \label{eq:h_complex_left_inv}
            \left( \Wedge^\bullet \mathfrak v^*, \dbarb\right) \hookrightarrow \left( \mathscr C^\infty(G, \underline  \Lambda^{\bullet}), \dbarb\right),
    \end{equation}
    where we consider $\Wedge^q \mathfrak v^*$ as the subspace of left-invariant sections of $\underline \Lambda^q$. This subcomplex was shown in \cite{jahnke_elliptic_2023} to induce an inclusion
    \begin{equation}
        \label{eq:the_homomorphism}
         H^{q} ( \mathfrak v, \C) \to H^{q} (G, \mathfrak v).
    \end{equation}
    
    The following consequence of Theorem~\ref{thm:main} answers a question proposed in \cite{jahnke_top-degree_2019, jahnke_elliptic_2023} and is part of a larger research project that includes \cite{jahnke_closed_2026, araujo_real_2026} and references therein.

    \begin{cor}
    \label{cor:left-invariant cohomology}
        Let $G$ be a connected, compact Lie group with a CR structure $\mathfrak{v}$ of hypersurface type. If the toral part $\mathfrak m$ satisfies \DC, then $H^q(G, \mathfrak v)$ can be computed by left-invariant forms, that is, the map \eqref{eq:the_homomorphism} is an isomorphism. In particular, we have $
            H^q(G, \mathfrak v) \cong \Wedge^q \mathfrak m^*.$
    \end{cor}

    As we will see in Example \ref{exa:SU(2)}, Corollary \ref{cor:left-invariant cohomology} no longer holds if one removes the \DC condition. More precisely, the Charbonnel--Khalgui classification splits the possible left-invariant CR structures of hypersurface type on $G$ into two classes, which are referred to as CR0 and CR1 (Section \ref{sec:CR structures}), and the \DC condition precludes the CR1 structures (Theorem \ref{thm:nodc_nocr1}). We show that, for structures of type CR0, \DC is also necessary for finite-dimensionality of the total cohomology space.
    
    \begin{thm}
    \label{thm:necessity}
        Let $G$ be a connected, compact Lie group with a left-invariant $\CRO$ structure $\mathfrak v$ of hypersurface type. Then, the total cohomology space
        \[
        H^\bullet(G,\mathfrak v) = \bigoplus_{q\geq 0}H^q(G,\mathfrak v)
        \]
        is finite-dimensional if and only if the toral part $\mathfrak m$ of $\mathfrak v$ satisfies \DC.
    \end{thm}

    It is natural to ask whether the cohomology of a CR1 structure is also necessarily infinite-dimensional. The methods of this paper rely on the CR0 decomposition and do not apply to the CR1 case, which requires substantially different techniques and will not be addressed here.

    The paper is organized as follows: In Section \ref{sec:CR structures}, we recall the notion of involutive structures and the construction of the associated differential complex and cohomology spaces. We also discuss the special case of left-invariant involutive structures on compact Lie groups and state the classification theorem by Charbonnel--Khalgui \cite{charbonnel_classification_2004}. We also provide an overview of the proof of Theorem \ref{thm:main} to motivate the technical results of the subsequent sections.
    Section \ref{sec:cohomology_of_lie_algebras} reviews Lie algebra cohomology and two notions of cohomologies induced by subalgebras, which will be the main tools in the proof of Theorem \ref{thm:main}. In Section \ref{sec:fourier_transform}, we introduce the Fourier transform framework used in several proofs, in particular, the connection between the Fourier transform on the Lie group and functions on the quotient by a maximal torus. Section \ref{sec:semilocal} examines the structure of the CR vector fields in a full neighborhood of an invariant torus.
    In Section \ref{sec:the_proof}, we apply a result of Hochschild–Serre on Lie algebra cohomology, together with the results from the previous sections, to prove Theorem \ref{thm:main}. Finally, in Section \ref{sec:necessity_proof}, we establish Theorem \ref{thm:necessity}.

\section{CR structures}\label{sec:CR structures}

    First, we establish some notation and conventions. Throughout the text, $G$ always denotes a connected, compact Lie group of dimension $N$. We denote by $\mathfrak g$ the Lie algebra of $G$. Let $L_g \colon G \to G$ be the left-multiplication by $g \in G$, that is, $L_g(h) = g \cdot h$. Notice that the push-forward $(L_g)_*$ induces a vector bundle isomorphism between $G \times  \mathfrak g_\C$ and $T_\C G$. Therefore, we can identify subbundles of $ T_\C G$ with subbundles of the trivial bundle $G  \times  \mathfrak g_\C$.
    
    A subbundle $\mathcal V \subset  T_\C G$ is said to be left-invariant if $(L_g)_* X \in \mathcal V_{g \cdot h}$ for all $X \in \mathcal V_h.$ We identify left-invariant involutive subbundles of $T_\C G$ with complex subalgebras of $ \mathfrak g_\C$. For a complex subalgebra $\mathfrak v \subset  \mathfrak g_\C$, we denote by $\mathcal V$ the associated subbundle. Notice that the rank of the bundle $\mathcal V$ is the complex dimension of $\mathfrak v$.

    A complex subalgebra $\mathfrak v \subset \mathfrak g_\C$ is a \textit{CR subalgebra} if $\mathfrak v \cap \overline{\mathfrak v} =\{0\}$. In this case, the associated subbundle $\mathcal V$ defines a left-invariant CR structure on the group $G$. We say that $\mathfrak v$ is of \textit{hypersurface type} if $\dim_\C \mathfrak v = [N/2]$\footnote{For $x \in \R$, we define $[x] = \max \{ n \in \Z \, | \, n \leq x \}$.}. This decomposes the problem into two natural cases, depending on the parity of $N$: If $N=2k+1$ is odd, then $\dim_\C \mathfrak v = k$ and we get, locally, the structure of a real hypersurface in $\C^{k+1}$. If $N=2k$ is even, then $\dim_\C \mathfrak v = k$ and $\mathfrak v\oplus \overline{\mathfrak v}=\mathfrak g_\C$, so $\mathfrak v$ defines a complex structure on $G$ (see \cite[Section I.8]{berhanu_introduction_2008}). Since we are mainly interested in CR structures, we will assume that $G$ is odd-dimensional throughout, but similar arguments apply to recover the even-dimensional (complex) case.
    
    \subsection{Left-invariant CR structures of hypersurface type on compact Lie groups}\label{subsect: CR structures on hypersurface type}
        In this section, we give a brief exposition of some results from \cite{charbonnel_classification_2004}. Let $\mathfrak t \subset \mathfrak g$ be a maximal abelian subalgebra (in particular, a Cartan subalgebra \cite[Lemma 12.2.1]{hilgert_structure_2012}). We denote by $\Delta$ the set of roots of $ \mathfrak t_\C$ in $ \mathfrak g_\C$ and by $\Delta_+$ a maximal subset of positive roots of $\Delta$. For $\alpha \in \Delta$, we denote by $\mathfrak g_\alpha$ the eigenspace associated to $\alpha$. Each $\mathfrak{g}_\alpha$ is one-dimensional \cite[Theorem 7.23]{hall_lie_2015} and
        \begin{equation}
        \label{eq:positive_roots_algebra}
            \mathfrak b_{\mathfrak t} = \bigoplus_{\alpha \in \Delta_+} \mathfrak g_\alpha
        \end{equation}
        is a subalgebra of $ \mathfrak g_\C$. Since $\mathfrak{t}$ is abelian, any choice of vector space $\mathfrak{m} \subset \mathfrak{t}_\C$ is a subalgebra, and $\mathfrak{h} = \mathfrak m \oplus \mathfrak{b}_\mathfrak{t}$ is a subalgebra of $\mathfrak g_\C$. Also, notice that $\mathfrak{b}_\mathfrak{t}$ is an ideal of $\mathfrak h$.
        
        Let us assume that $\mathfrak{m}$ is a CR subalgebra, in which case $\mathfrak{h}$ is also a CR subalgebra. We denote by $d$ the dimension of $\mathfrak t$ and, since the dimension of each $\mathfrak g_\alpha$ is one, we can easily see that $N = d + 2l$ with $l$ being the number of elements of $\Delta_+$. If $\mathfrak m$ is a subalgebra of $ \mathfrak t_\C$ of dimension $[d/2]$ such that $ \mathfrak m \cap \mathfrak g = \{0\}$, then the sum $\Phi(\mathfrak m) =\mathfrak m \oplus \mathfrak b_{\mathfrak t}$ is a subalgebra of dimension $[N/2]$ with $ \Phi(\mathfrak m) \cap \mathfrak g = \{0\}$.
        
        We recall that two subalgebras $\mathfrak h_1$ and $\mathfrak h_2$ of $\mathfrak g$ are called conjugate if there exists a $g \in G$ such that $\operatorname{Ad}(g) \mathfrak h_1 = \mathfrak h_2$. Here $\operatorname{Ad}(g)$ is the differential at $e \in G$ of the map $C_g \colon G \to G$ given by $C_g(h) = ghg^{-1}$. 
    
        \begin{dfn}
            A subalgebra $\mathfrak h \subset  \mathfrak g_\C$ is said to be of type $\CRO$ (with respect to $\mathfrak g \subset \mathfrak g_\C$) if it is conjugate to a subalgebra of the form
            \begin{equation}
            \label{eq:CR0}
                \Phi(\mathfrak m) = \mathfrak m \oplus \mathfrak b_{\mathfrak t}, 
            \end{equation}
            where $\mathfrak m$ is a subspace of dimension $[d/2]$ of $\mathfrak t_\C$ with $\mathfrak m \cap \mathfrak g = \{0\}$.
        \end{dfn}
    
        When $N$ is odd, we introduce the set $\mathcal D(\Delta_+)$ of all elements of the form $(\alpha, \mathfrak m, x, t)$ with $\alpha$ a simple root in $\Delta_+$, $\mathfrak m \subset \ker \alpha$ a subspace of dimension $[d/2]$ with $\mathfrak m \cap \mathfrak g = \{0\}$, and such that $x$ is an element of $\mathfrak g_\alpha \backslash \{0\}$ and $t \in \mathfrak t \backslash \{0\}$. For every $(\alpha, \mathfrak m, x, t)$ in  $\mathcal D(\Delta_+)$ we define
        \begin{equation}
        \label{eq:CR1}
            \Theta(\alpha, \mathfrak m, x, t) = \mathfrak m \oplus \bigoplus_{\beta \in \Delta_+ \backslash \{\alpha\}} \mathfrak g_\beta \oplus \operatorname{span}_\C\{t+x\}.
        \end{equation}
        
        Therefore, the space $\Theta(\alpha, \mathfrak m, x, t)$ is a subalgebra of $ \mathfrak g_\C$ of dimension $[N/2]$ with $\Theta(\alpha, \mathfrak m, x, t) \cap \mathfrak g = \{0\}$.
    
        \begin{dfn}
            A subalgebra $\mathfrak h \subset  \mathfrak g_\C$ is said to be of type $\CRI$ if it is conjugate to a subalgebra of the form $\Theta(\alpha, \mathfrak m, x, t)$, where $(\alpha, \mathfrak m, x, t)$ is an element of $\mathcal D(\Delta_+)$.
        \end{dfn}
        
        Without loss of generality, we take $\mathfrak h$ to be of the form \eqref{eq:CR0} or \eqref{eq:CR1} instead of conjugate to this form. For an algebra $\mathfrak{h}$ of type $\CRO$ or $\CRI$, we call the subalgebra $\mathfrak{m}$ the toral part of $\mathfrak{h}$.
    
        The main theorem of \cite{charbonnel_classification_2004} is the following:
        
        \begin{thm}\label{thm:classification_theorem}
            Let $\mathfrak{v}\subset \mathfrak{g}_{\C}$ be a left-invariant CR structure of hypersurface type on an $N$-dimensional compact Lie group $G$. If $N$ is even, then $\mathfrak v$ is a complex structure of type $\CRO$. If $N$ is odd, then $\mathfrak v$ is either of type $\CRO$ or of type $\CRI$.
        \end{thm}

        From this point on, we assume that $G$ is odd-dimensional without further mention.

    \begin{rmk}
    
    By using the CR0 or CR1 decomposition, one can show that a CR Lie algebra of hypersurface type is Levi-degenerate if $\mathfrak m \neq \{0\}$. The only compact Lie group that can have a non-degenerate left-invariant CR structure of hypersurface type is $\operatorname{SU}(2)$ with the standard CR structure. 
    
    \end{rmk}

    \subsection{The division condition \DC}
    \label{sec:division_condition}
    
    Let $\mathcal{V}$ be a left-invariant CR structure on $G$ with associated Lie algebra $\mathfrak v$ and toral part $\mathfrak m \subset \mathfrak t_\C$.
    Let $L \in \mathfrak t_\C$, $d = \dim \T = \R^d/\Z^d$, and let $(x_1, \dots, x_d)$ be the standard coordinates on $\T$. Then $L$ has the form $L = \sum_{j=1}^d a_{j} \frac{\partial}{\partial x_j}$ with $a_{j} \in \C$. The symbol of $L$ is defined to be
    $$\widehat{L}(\xi) = i \sum_{j=1}^d a_{j}\xi_j$$ for $\xi =(\xi_1, \dots, \xi_d) \in \Z^d$. We say that $\mathfrak{m}$ satisfies the division condition \DC if there exists a basis $\{L_1, \ldots, L_r\}$ of $\mathfrak m$ and constants $C, \rho > 0$ such that
    $$
        \max_j |\widehat{L_j}(\xi)| \geq C(1 + |\xi|)^{-\rho}, \quad \forall \xi \in \mathbb{Z}^d \setminus \{0\}.
    $$

    The \DC condition generalizes a condition from \cite{greenfield_global_1972} that characterizes the regularity of solutions to a partial differential equation and has been extensively studied in relation to solvability and regularity. For further results on the division condition and its generalizations in systems of partial differential equations, see \cite{araujo_global_2024} and references therein.

    The class of subalgebras of type $\CRO$ provides many examples of CR structures whose toral part satisfies the \DC condition. In fact, for every algebra $\mathfrak m \subset \mathfrak t_\C$ satisfying \DC, we get a CR structure of type $\CRO$ whose toral part satisfies \DC by setting
    \[
        \mathfrak v \doteq \mathfrak m \oplus \bigoplus_{\alpha \in \Delta_+} \mathfrak g_\alpha.
    \]
     \begin{exa} \label{exa:nice_little_example}
    Consider a 3-torus $\T = \R^3 / \Z^3$ with coordinates $(x,y,t)$. For $\lambda \in \R$ consider the vector field 
    \[
    L_{\lambda} = \lambda  \frac{\del}{\del x}  + \frac{\del}{\del y} + i \frac{\del}{\del t}
    \]
    on $\T$. The algebra $\mathfrak m_\lambda \doteq \operatorname{span} \{L_{\lambda}\}$ defines a CR structure on $\T$ and satisfies $\DC$ if and only if $\lambda \in \R \setminus \Q$ is not a Liouville number. As described above, we can now use $\mathfrak m_\lambda$ to endow either $\T^2 \times \SU(2)$ or $\SU(4)$ with a CR structure of type $\CRO$ whose toral part satisfies \DC depending on the choice of $\lambda$.
    \end{exa}

    Our first result shows that no Lie algebra of type $\CRI$ can have a toral part satisfying \DC. 

    \begin{thm}
        \label{thm:nodc_nocr1}
        If $\mathfrak v$ is of type $\CRI$, then the toral part of $\mathfrak v$ cannot satisfy $\DC$.
    \end{thm}

    \begin{proof}
        If $\mathfrak v$ is of type $\CRI$, then we can write $$\mathfrak v = \mathfrak m \oplus \bigoplus_{\alpha \in \Delta_+ \backslash \{ \beta \}} \mathfrak g_\alpha \oplus \operatorname{span}_\C \{t + x\}$$ with $\beta$ a simple root, $\mathfrak m \subset \ker \beta$, $\mathfrak m \cap \mathfrak g = \{0\}$, $t \in \mathfrak t \backslash \{0\}$, and $x \in \mathfrak g_\beta \backslash \{0\}.$

        Since $\beta \in \mathfrak t_\C^*$ is purely imaginary, we have that $\mathfrak m\oplus \overline{\mathfrak m} \subset \ker \beta.$ Now notice that $\mathfrak k = (\mathfrak m \oplus \overline{\mathfrak m}) \cap \mathfrak t$ is a real Lie algebra, which we can integrate into a Lie group $K = \exp_\T(\mathfrak k)\subset \T$.
        
        If $\mathfrak m$ satisfies \DC, then $K$ cannot be closed. In fact, if $K$ were closed, then $\T/K$ would be a Lie group of positive dimension, and any smooth function on $\T/K$ would lift to a function on $\T$ that is constant along the cosets of $K$, hence annihilated by $\mathfrak k_\C$ and in particular by $\mathfrak m$. This would produce non-constant CR functions on $\T$, contradicting the fact that under \DC every CR function on $\T$ is constant.

        Now, by taking left-translations, the root $\beta$ can be seen as a 1-form vanishing over $K$. If $K$ is not closed, since $\dim \mathfrak k = \dim \mathfrak t-1$, we have that $K$ is dense in $\T$ and thus $\beta$ vanishes in $\T$, contradicting the fact that $\beta$ is a root.
    \end{proof}

    \subsection{The tangential CR complex}
    \label{sec:involutive_structures}

    This section briefly recalls the construction of the differential complex for CR and complex structures; see \cite{berhanu_introduction_2008, boggess_cr_1991} for details.

    Let $\Omega$ be a smooth and orientable manifold of dimension $N$. We recall that a CR structure on $\Omega$ is a smooth subbundle $\mathcal V$ of $T_\C \Omega$ such that $\mathcal V \cap \overline{\mathcal V} = \{0\}$ and the Lie bracket of any two smooth local sections of $\mathcal V$ is again a smooth section of $\mathcal V$. In general, if $W$ is a smooth vector bundle, we denote by $ \mathscr C^\infty(\Omega, W)$ the space of sections of $W$ with smooth coefficients. 

    Let $\Ann \mathcal V \subset  T^*_\C \Omega$ be the annihilator of $\mathcal V$ and define $\underline \Lambda^{q} \doteq \Wedge^{q} ( T^*_\C \Omega / \Ann \mathcal V)$, then there is a unique differential map
    \begin{equation}
    \label{eq:dprime}
        \dbarb \colon \mathscr C^\infty(\Omega, \underline \Lambda^{q}) \to \mathscr C^\infty(\Omega, \underline \Lambda^{q+1})
    \end{equation}
    making the diagram
    \begin{equation}
        \label{eq:diagram}
        \begin{tikzcd}
            \mathscr C^\infty(\Omega, \Wedge^q  T^*_\C \Omega) \arrow[r, "\dext"] \arrow[d, "\pi_q"] & \mathscr C^\infty(\Omega, \Wedge^{q+1}  T^*_\C \Omega) \arrow[d, "\pi_{q+1}"] \\
            \mathscr C^\infty(\Omega, \underline \Lambda^{q}) \arrow[r, "\dbar_b"]           & \mathscr C^\infty(\Omega, \underline \Lambda^{q+1})         
        \end{tikzcd}
    \end{equation}
    commute. Here, $\dext$ is the usual exterior derivative for differential forms on $\Omega$.
     
    Since $\dext \circ \dext = 0$, we get that $\dbarb \circ \dbarb = 0$ and thus $\dbarb$ defines a differential complex. For each $q$, the associated cohomology spaces are defined as
    \begin{equation*}
      H^{q} (\Omega, \mathcal V)
      \doteq
      \frac{ \operatorname{ker} \{ \dbarb \colon \mathscr C^\infty(\Omega, \underline \Lambda^{q}) \to \mathscr C^\infty(\Omega, \underline \Lambda^{q + 1}) \} }
           { \operatorname{ran} \{ \dbarb \colon \mathscr C^\infty(\Omega, \underline \Lambda^{q - 1}) \to \mathscr C^\infty(\Omega, \underline \Lambda^{q}) \} }.
    \end{equation*}
    Throughout the text, we will also consider the cohomology spaces $H^{q}_{} (U, \mathcal V)$ defined by restricting the $\dbarb$-complex to an open subset $U\subset \Omega$.
    
    \begin{rmk}
       In the complex case, i.e.,  $\mathcal V \oplus \overline{\mathcal V} = T_\C \Omega$, the differential map $\dbar_b$ coincides with the usual complex operator $\dbar$ and $H^q(\Omega, \mathcal{V})$ is the Dolbeault cohomology group $H^{0,q}_{\dbar}(\Omega)$.
    \end{rmk}
     
    We will usually be in the situation where $\mathcal{V}$ is a left-invariant CR structure with corresponding Lie algebra $\mathfrak v$ on a compact Lie group $G$. In this case, we denote the space $H^{q} (\Omega, \mathcal V)$ by $H^{q}(G, \mathfrak v)$. Suppose the left-invariant CR structure $\mathcal{V}$ is of hypersurface type. Then, as we saw in the previous sections (see Theorem \ref{thm:classification_theorem}), there exists a maximal torus $\T \subset G$ with Lie algebra  $\mathfrak t$ and a Lie algebra $\mathfrak m \subset \mathfrak t_\C$ defining a bi-invariant CR structure on $\T$ such that $\mathfrak v = \mathfrak m \oplus \mathfrak b$ with $\mathfrak b$ a Lie algebra.

    Consider the $\dbar_b$-complex on $\mathbb T$ with respect to $\mathfrak{m}$. If $\mathfrak m$ satisfies \DC, it follows from \cite[Lemma 3]{jacobowitz_levi-flat_2023} that we have an injective map
    \[
         H^{q}(\mathbb T, \mathfrak{m}) \hookrightarrow  H^{q}(G, \mathfrak{v}).
    \]
    Let us briefly recall how this map is constructed: if \DC is satisfied, then every element of $H^q(\mathbb T,\mathfrak m)$ has a left-invariant representative, which we can extend by zero to identify with an element $u\in \mathfrak g^\ast_\C$. Then, we simply extend this element to $G$ via left-translation, $w_g = (L_{g^{-1}})^\ast u$, and check that the resulting object defines a class in $H^q(G,\mathfrak v)$. For more details, see  \cite[Lemma~3]{jacobowitz_levi-flat_2023}.
    
    The following provides an example that the cohomology groups $H^q(G, \mathfrak v)$ are not necessarily finite-dimensional if the toral part of $\mathfrak v$ does not satisfy \DC.
    \begin{exa}\label{exa:SU(2)}
   Consider the Lie group $\SU(2)$. Its Lie algebra $\mathfrak{su}(2)$ is generated by the matrices
        $$ X = \left(  \begin{matrix} 0 & i \\ i & 0 \end{matrix} \right), \quad Y = \left(  \begin{matrix} 0 & -1 \\ 1 & 0 \end{matrix} \right), \quad  T = \left(  \begin{matrix} i & 0 \\ 0 & -i \end{matrix} \right).$$ It is easy to verify that these matrices satisfy the following bracket relations:
        \begin{equation}
        [T, X] = 2Y, \quad [T, Y] = -2X, \quad [X, Y] = 2T.
        \end{equation}
        Moreover, the group $\SU(2)$ can be identified with the unit sphere $\mathbb{S}^3$ in $\C^2.$ The complex structure from $\C^2$ defines a CR structure on $\mathbb{S}^3$. This structure can be pulled back to $\SU(2)$, and it is easy to verify that it is left-invariant of type $\CRO$, and generated by the vector $L = X - iY$ (cf. \cite[pg 498]{jahnke_elliptic_2023}). The toral part of this CR structure vanishes and thus does not satisfy \DC. Now, the cohomology group $H^0(SU(2), \mathfrak v)$ is infinite-dimensional since the restriction of every holomorphic polynomial in $\C^2$ to $\mathbb{S}^3$ defines a CR function of the CR manifold $(\SU(2), \mathfrak v)$ (cf. \cite[Chapter 9]{boggess_cr_1991}).
    \end{exa}

    \subsection{Overview of the main arguments} We shall briefly sketch the ideas involved in the proof of the main results. The first step is an algebraic reformulation of the cohomology space $H^{q}(G,\mathfrak v)$ (which is, a priori, an analytic object coming from the $\dbar_b$-complex). Applying the language of Lie algebra cohomology and well-known identification of complexes, we can obtain a canonical isomorphism $H^{q}(G,\mathfrak v) \cong H^q(\mathfrak v, \mathscr C^\infty(G))$, where the latter space is the cohomology of the Lie algebra $\mathfrak v$ (which defines the CR structure) with values in the infinite-dimensional module $\mathscr C^\infty(G)$ of smooth $\C$-valued functions on $G$. At this point, we perform a Fourier decomposition of $\mathscr C^\infty(G)$, which correspondingly yields a decomposition of the cohomology
    \begin{equation}
        \label{eq:decomp_intro}
        H^q(\mathfrak v, \mathscr C^\infty(G)) \cong \widehat \bigoplus_{[\xi] \in \widehat{G}}H^q(\mathfrak v, \End(H_\xi)),
    \end{equation}
    where $\widehat{G}$ is the usual space of (equivalence classes of) continuous, unitary representations of $G$ and $H_{\xi}$ is a finite-dimensional vector space. The goal is to show that, for non-trivial $\xi$, the summand $H^q(\mathfrak v, \End(H_\xi))$ vanishes. Applying the classification result of Charbonnel--Khalgui (Theorem \ref{thm:classification_theorem}), we can decompose  $\mathfrak v = \mathfrak m \oplus \mathfrak b$, where $\mathfrak m \subset \mathfrak t_{\C}$ is a bi-invariant CR structure on a maximal torus $\mathbb T\subset G$ and, under \DC, $\mathfrak b$ is an ideal. 
    Since $\mathfrak{m}$ acts completely reducibly on $\End(H_\xi)$, we can apply a suitable generalization of a theorem of Hochschild--Serre \cite{hochschild_cohomology_1953} and obtain an isomorphism
    \[
    H^q(\mathfrak v, \End(H_\xi)) \cong \bigoplus_{j+k=q}H^j(\mathfrak m, \C)\otimes H^k(\mathfrak b, \End(H_\xi))^\mathfrak m.
    \]
    At this point, \DC allows us to show that $H^k(\mathfrak b, \End(H_\xi))=0$, which ultimately comes from a semilocal calculation performed in a tubular neighborhood of the maximal torus (cf. Lemma \ref{lem:invariance_functions}). 
    Finally, we use standard results on the elliptic complex given by the operator $\dbar$ on the compact complex manifold $G/\mathbb T$ to extract an isomorphism $H^q(G, \mathfrak v)\cong H^q(\mathfrak m, \C)$ from the isomorphism \eqref{eq:decomp_intro}, which implies that $H^q(G, \mathfrak v)$ is finite-dimensional. Now, the statements in Theorem \ref{thm:main} and Corollary \ref{cor:left-invariant cohomology} follow from dimensionality reasons.

    Finally, for Theorem \ref{thm:necessity}, it remains to show that, for $\CRO$ structures, failure of $\DC$ implies infinite-dimensional cohomology. The idea is to apply the theory of weights: assuming there is a sequence $(\xi_N)$ of Fourier frequencies in which the symbol of the CR vector fields induced on the maximal torus becomes very small, one applies elements of the Weyl group of $\T\subset G$ in order to realize them by dominant weights. Then, one can find a sequence of representations of $G$ which have these as highest weights, and by the general formalism, they vanish when differentiated by vectors in the positive root space $\mathfrak{b}_{\mathfrak t}$. Therefore, the CR vector fields act on these functions only by the toral directions. 

    When in this situation, one proceeds classically (see for example \cite{araujo_real_2026}) and can show that, either there are infinitely many independent global CR functions (hence $H^0(G,\mathfrak v)$ is infinite-dimensional), or $H^1(G,\mathfrak v)$ is non-Hausdorff (and therefore infinite-dimensional) by explicitly constructing a $1$-form as an infinite series which is not exact, but has each partial sum exact. This concludes the proof.

    \section{Lie algebra cohomology}
    \label{sec:cohomology_of_lie_algebras}

        This section summarizes the basic definitions and results of Lie algebra cohomology. For a full introduction, see \cite[Chapter 7.5]{hilgert_structure_2012} and \cite{hochschild_cohomology_1953}. 

        Let $\mathfrak{h}$ be a (complex) Lie algebra and let $M$ be a $\mathfrak{h}$-module, i.e., a vector space equipped with a Lie algebra homomorphism $\rho\colon\mathfrak{h} \to \mathfrak{gl}(M)$. For $X \in \mathfrak{h}$ and $x \in M$, we denote $\rho(X)x$ by $X \cdot x$. The invariant submodule is defined as
        \[
        M^{\mathfrak h} = \{ x \in M \, | \, \textup{$X \cdot x =0$ for all $X \in \mathfrak h$} \}.
        \]
        Following the notation of \cite{hochschild_cohomology_1953}, we denote by $C^p(\mathfrak{h}, M)$ the vector space of all alternating multilinear $p$-forms on $\mathfrak{h}$ with values in $M$, and let $C^\bullet(\mathfrak{h}, M) = \bigoplus_{p \geq 0} C^p(\mathfrak{h}, M)$ be the direct sum of all $C^p(\mathfrak{h}, M)$. 

        For any $p \in \N$, let $u$ be an alternating $p$-linear form on $\mathfrak{h}$ with values in $M$, and let $X_1, \ldots, X_{p+1} \in \mathfrak{h}$. We define the operator $\dext\colon C^p(\mathfrak{h}, M) \to C^{p+1}(\mathfrak{h}, M)$ by
        \begin{equation}
        \label{eq:algebraic_diff_op}
            \begin{aligned}
                 \dext u(X_1, \ldots, X_{p+1}) & = \sum_{j=1}^{p+1} (-1)^{j+1} X_j \cdot u(X_1, \ldots, \widehat{X}_j, \ldots, X_{p+1}) \\
                    & + \sum_{j < k} (-1)^{j+k+1} u([X_j, X_k], X_1, \ldots, \widehat{X}_j, \ldots, \widehat X_k, \ldots, X_{p+1}),
            \end{aligned}
        \end{equation}
        in which $\widehat{X}_j$ denotes the omission of $X_j$.
            
        An easy calculation shows that $\dext$ satisfies $\dext \circ \dext = 0$, so we have a cochain complex $(C^\bullet(\mathfrak{h}, M), \dext)$. The cohomology spaces of this complex are denoted by $H^\bullet(\mathfrak{h}, M)$. The following is an important example that will be further expanded in the text.

        \begin{exa}
            Let $G$ be a compact Lie group with its complexified Lie algebra $\mathfrak g_\C$. By identifying $\mathfrak g_\C$ with complex left-invariant vector fields on $G$, we can endow $\mathscr C^\infty(G)$ with a structure of $\mathfrak g_\C$-module. It is straightforward to verify that $H^\bullet(\mathfrak{g}_\C, \mathscr C^\infty(G))$ corresponds to the usual de Rham cohomology of $G$ with complex coefficients.
        \end{exa}

        If $\mathfrak{b}$ is an ideal of $\mathfrak{h}$ and $M$ is a $\mathfrak h$-module, we can define a $\mathfrak{h}$-module structure on the spaces $C^p(\mathfrak{b}, M)$. In fact, $C^0(\mathfrak{b}, M) = M$ is already a $\mathfrak{h}$-module, and for $p > 0$, $u \in C^p(\mathfrak{b}, M)$, $X \in \mathfrak{h}$, and $Y_1, \ldots, Y_p \in \mathfrak{b}$, we define the $\mathfrak{h}$-action on $u$ by 
        \begin{equation}
        \label{eq:lie_derivative}
            (\mathcal{L}_X u)(Y_1, \ldots, Y_p) = X \cdot u(Y_1, \ldots, Y_p) - \sum_{i=1}^p u(Y_1, \ldots, [X,Y_i], \ldots, Y_p).
        \end{equation}
                        
        We have $\mathcal L_X u \in C^p(\mathfrak b, M)$ for all $u \in C^p(\mathfrak b, M)$ and that $\mathcal L$ depends linearly on $X\in \mathfrak h$ and on $u$.

       Recall that a module is called \textit{completely reducible} if it is the direct sum of irreducible submodules. Further, a subalgebra $\mathfrak m \subset \mathfrak h$ is called reductive in $\mathfrak h$ if the adjoint action of $\mathfrak m$ on $\mathfrak h$ is completely reducible.
       
       The following generalization of \cite[Theorem 13]{hochschild_cohomology_1953} (see also \cite[Theorem 4.1]{hochschild_differential_1962} and \cite[Theorem 2.28]{solleveld}) is used crucially in the proof of Theorem \ref{thm:main} given in Section \ref{sec:the_proof}. 
       For the reader’s convenience, we restate the theorem in our notation.
        
        \begin{thm}
        \label{thm:hsc_thm_13}
            Let $\mathfrak h$ be a finite-dimensional complex Lie algebra, and let $M$ be a finite-dimensional $\mathfrak h$-module. Suppose that we can write $\mathfrak h = \mathfrak m \oplus \mathfrak b$ as a direct sum of subalgebras, where $\mathfrak b \subset \mathfrak h$ is an ideal, $\mathfrak m$ is reductive in $\mathfrak h$, and $M$ is completely reducible as an $\mathfrak m$-module. Then,
            \[
                H^n(\mathfrak h, M) \cong \sum_{i+j=n} H^i(\mathfrak m, \C) \otimes H^j(\mathfrak b, M)^{\mathfrak m},
            \]
            for all $n \geq 0$.
        \end{thm}
        The space $H^j(\mathfrak b, M)^{\mathfrak m}$ in the above theorem is defined with respect to the $\mathfrak m$-module structure on $H^j(\mathfrak b, M)$ given by the $\mathfrak m$-action on $C^j(\mathfrak b, M)$ in \eqref{eq:lie_derivative}.
        
        \subsection{The cohomology induced by subalgebras}
        \label{sec:cohomology_induced_subalgebra}
            
            In this section, we review some basic concepts of Lie algebra cohomology induced by a subalgebra. We note that there are different definitions of Lie algebra cohomology with respect to subalgebras. In this paper, we make use of \emph{two} different notions. First, we introduce the notion corresponding to the algebraic version of the cohomology associated with left-invariant involutive structures. This is the algebraic counterpart of the construction of the operator $\dbarb$ \eqref{eq:diagram}. The reader is referred to \cite[Chapter VII.4]{berhanu_introduction_2008} for further details.
        
            Let $\mathfrak f$ be a subalgebra of $\mathfrak h$. To define the cohomology induced by $\mathfrak f$, for $q \geq 0$, 
            let $N^{q}_{\mathfrak f}(\mathfrak h, M)$ be the set of all $u \in C^{q}(\mathfrak h, M)$ that vanish on $\Wedge^q\mathfrak f$. By convention, if $q < 0$, we set $N^{q}_{\mathfrak f}(\mathfrak h, M) = {0}$. Notice that the inclusion 
            \begin{equation}
            \label{eq:relative}
                \dext N^{q}_{\mathfrak f}(\mathfrak h, M) \subset N^{q+1}_{\mathfrak f}(\mathfrak h, M),
            \end{equation} holds for all $q$.
            We define the quotient space
            $ C^{q}_{\mathfrak f}(\mathfrak h, M) \doteq C^q(\mathfrak h, M) / N^{q}_{\mathfrak f}(\mathfrak h, M)$. From \eqref{eq:relative}, we have an induced map
            $$ \dext' \colon C^{q}_{\mathfrak f}(\mathfrak h, M) \to C^{q+1}_{\mathfrak f}(\mathfrak h, M).$$
        
            Thus, we obtain a cochain complex $(C^\bullet_{\mathfrak f}(\mathfrak h, M), \dext')$ whose cohomology spaces are denoted by $H^{\bullet}_{\mathfrak f}(\mathfrak h, M)$.
            We will be mostly interested in the case where $\mathfrak h = \mathfrak g_\C$ is the complexified Lie algebra of a compact Lie group $G$, and $ \mathfrak f =\mathfrak v$ is the subalgebra associated to a left-invariant CR structure $\mathcal{V}$ on $G$.
            The following describes how to relate the cohomology of the $\dbar_b$ operator and the cohomology induced by the subalgebra $\mathfrak v \subset \mathfrak g_\C$: Notice that there is a canonical isomorphism
                \begin{align*}
                    T_q\colon \mathscr C^\infty(G,\Wedge^qT^*_\C G) &\rightarrow C^q(\mathfrak g_{\C}, \mathscr C^\infty(G))
                \end{align*}
                defined in the following way (for $q\geq 1$, as in $q=0$ these spaces are identical):
                \begin{equation}
                    (T_q \omega)(v_1,\ldots,v_q)(g)\doteq \omega_g(\dext L_{g}v_1,\ldots,\dext L_{g}v_q),
                \end{equation}
                where $\omega \in \mathscr C^\infty(G, \Wedge^q T_\C^* G)$ is a $q$-form on $G$, $v_1,\ldots,v_q \in \mathfrak{g}_{\C}$, $g \in G$ and $L_g$ is the left-translation by $g$.
                It is easy to see that this map commutes with the differentials and maps the subspace $\{\omega \in \mathscr C^\infty(G,\Wedge^qT^*_\C G) \, | \, \,\omega(X_1,\ldots,X_q)=0\text{ for all }X_j \in \mathcal{V}\}$ isomorphically onto $N^q_{\mathfrak v}(\mathfrak g_\C, \mathscr C^\infty(G))$. Therefore, it induces an isomorphism of complexes
                \[
                \left(C^{\bullet}_{\mathfrak v}(\mathfrak g_\C, \mathscr C^\infty(G)), \dext' \right) \cong \left(\mathscr C^\infty(G, \underline \Lambda^{\bullet}),\dbar_b\right),
                \]
                which yields a canonical isomorphism $H^\bullet(G, \mathfrak{v})\cong H^{\bullet}_{\mathfrak v}(\mathfrak{g}_\C, \mathscr C^\infty(G))$. The following identification will play an important role in the proof of Theorem \ref{thm:main}.
            \begin{prp}\label{prp:iso_dbar_algebra}
            Let $G$ be a compact Lie group with its complexified Lie algebra $\mathfrak g_\C$. If $\mathfrak v \subset \mathfrak g_\C$ is a subalgebra defining a left-invariant CR structure on $G$, then we have
             \begin{equation*}
                H^{\bullet}(G, \mathfrak{v}) \cong H^{\bullet}(\mathfrak{v}, \mathscr C^\infty(G)).
            \end{equation*}
            \end{prp}
        \begin{proof}
            It was shown in \cite[Theorems~1 and 2]{hochschild_cohomology_1953}, that the previously constructed complexes $(C^{\bullet}_{\mathfrak v}(\mathfrak g, \mathscr C^\infty(G)), \dext')$ and $(C^\bullet(\mathfrak v, \mathscr C^\infty(G)), \dext)$ are isomorphic and thus we have $ H^{\bullet}_{\mathfrak v}(\mathfrak g, \mathscr C^\infty(G)) = H^{\bullet}(\mathfrak v, \mathscr C^\infty(G))$. By the previous discussion this also shows $H^{\bullet}(G, \mathfrak{v}) \cong H^{\bullet}(\mathfrak{v}, \mathscr C^\infty(G))$.
        \end{proof}
            
        \subsection{The cohomology relative to a subalgebra}
        
            In addition to the cohomology associated with left-invariant involutive structures, we introduce a second notion of cohomology. This notion is essential for the application of the Hochschild--Serre spectral sequence.
            
            Again, let $\mathfrak f$ be a subalgebra of $\mathfrak h$.
            We define $B^q(\mathfrak h, \mathfrak f; M)$ as the subspace of $C^q(\mathfrak h,  M)$ annihilated by the interior product $i_X$ for every $X \in \mathfrak f$. Note that $B^q(\mathfrak h,\mathfrak f;M)$ is stable under the action of $\mathfrak f$, so we define $C^q(\mathfrak h, \mathfrak f; M) = B^q(\mathfrak h, \mathfrak f; M)^{\mathfrak f}$ as the invariant elements. Since $\dext C^q(\mathfrak h, \mathfrak f;M) \subset C^{q+1}(\mathfrak h,\mathfrak f; M)$, we obtain a complex  $(C^\bullet(\mathfrak h, \mathfrak f; M), \dext)$ whose cohomology spaces are denoted by $H^\bullet(\mathfrak h, \mathfrak f; M)$. We refer to $H^\bullet(\mathfrak h, \mathfrak f; M)$ as the cohomology of $\mathfrak h$ relative to $\mathfrak f$ with values in $M$.
            
            We will usually use this notion to describe the cohomology of quotients of compact Lie groups. For instance, if $K$ is a connected, closed subgroup of $G$, $\mathfrak k \subset \mathfrak g$ the Lie algebra of $K$ and $\Omega \doteq G/K$, then it follows from \cite[p.~211 (2.7)]{bott_homogeneous_1957} that
            \[H^\bullet(\Omega, \C) = H^\bullet(\mathfrak g_\C, \mathfrak k_\C; \mathscr C^\infty(G)).\] 
            In the CR setting, there is an analogous statement for the Dolbeault cohomology of $\Omega$.

    \begin{prp}
    \label{prp:cohomology_coeff_C}
        Let $G$ be a compact, odd-dimensional Lie group equipped with a CR structure of type $\CRO$ represented by the Lie algebra $\mathfrak v \subset \mathfrak g_\C$. We write $\mathfrak v = \mathfrak m \oplus \mathfrak b$ and denote by $\mathfrak t_\C$ the complexification of the Lie algebra of the associated maximal torus $\T$. Let $\Omega = G/\T$ be equipped with the complex structure induced by $\pi_*(\mathfrak b)$. Then, there is a canonical isomorphism
        \[ H^{0,q}_{\dbar}(\Omega) \cong  H^q(\mathfrak t_\C \oplus \mathfrak b, \mathfrak t_\C; \mathscr C^\infty(G)), \qquad q \geq 0.\]
    \end{prp}

    \begin{proof}

        Let $\mathfrak u = \mathfrak t_\C \oplus \mathfrak b$ and consider $(C^\bullet(\mathfrak u, \mathscr C^\infty(G)), \dext)$. Since the complex structure on $\Omega = G / \T$ is defined by $\pi_*(\mathfrak u) = \pi_*(\mathfrak b)$, any smooth $(0,q)$-form $u$ on $\Omega$ pulls back to an element $\pi^*(u) \in C^q(\mathfrak u, \mathscr  C^\infty(G))$. We also have, by functoriality, that $\dext$ satisfies $\pi^* \circ \dbar = \dext \circ \pi^*$, where $\dext$ is the exterior derivative \eqref{eq:algebraic_diff_op} associated to the algebra $\mathfrak u$.
    
         The form $\pi^*(u)$ satisfies $i_X(\pi^*(u)) = 0$ for all $X \in \mathfrak t_\C$ since $\pi_*(X) = 0$ for all such $X$. Note, moreover, that an element $ \alpha \in B^q(\mathfrak u, \mathfrak t_\C, \mathscr C^\infty(G))$ is invariant under the $\T$-action if and only if $\mathcal{L}_X \alpha =0$ for all $X \in \mathfrak t$. In particular, we have $\mathcal{L}_X (\pi^{\ast}(u)) = 0$ for all $X\in \mathfrak{t}_\C$.

        Hence, $\pi^*(u)$ is an element of $B^q(\mathfrak u, \mathfrak t_\C, \mathscr C^\infty(G))$ and thus we obtain a morphism of complexes $\pi^* \colon \mathscr C^\infty (\Omega, \underline \Lambda^q) \to B^q(\mathfrak u, \mathfrak t_\C, \mathscr C^\infty(G))^{\mathfrak t_\C}$. We can explicitly construct an inverse in the following way: let $\omega \in B^q(\mathfrak u,\mathfrak t_\C;\mathscr C^\infty(G))^{\mathfrak t_\C}$, considered as an element of $C^q(\mathfrak u;\mathscr C^\infty(G))$. We can extend $\omega$ (by zero) to an element of $C^q(\mathfrak g_\C; \mathscr C^\infty(G))$. Since $i_X\omega=0$ for all $X\in \mathfrak t_\C$, $\omega$ actually descends to an element of $C^q(\mathfrak g_\C / \mathfrak t_\C;\mathscr C^\infty(G))$\footnote{Even though $\mathfrak{g}_\mathbb{C} / \mathfrak{t}_\mathbb{C}$ is not a Lie algebra, we denote by $C^q(\mathfrak{g}_\mathbb{C} / \mathfrak{t}_\mathbb{C};\mathscr{C}^\infty(G))$ the set of alternating multilinear $q$-forms on $\mathfrak{g}_\mathbb{C} / \mathfrak{t}_\mathbb{C}$ with values in $\mathscr{C}^\infty(G)$.}, and using $\mathcal{L}_X\omega = 0$, we descend further to $C^q(\mathfrak g_\C/\mathfrak t_\C;\mathscr C^\infty(\Omega))$. Composing with the projection onto the quotient $C^q(\mathfrak g_\C/\mathfrak{t}_\C;\mathscr C^\infty(\Omega))/N^q_{\mathfrak u}(\mathfrak g_\C/\mathfrak t_\C;\mathscr C^\infty(\Omega)) = C^q_{\mathfrak u}(\mathfrak g_{\C}/\mathfrak t_\C,\mathscr C^\infty(\Omega))$ yields a map
        \[
            \rho:B^q(\mathfrak u, \mathfrak t_\C, \mathscr C^\infty(G))^{\mathfrak t_\C} \to C^q_{\mathfrak u}(\mathfrak g_{\C}/\mathfrak t_\C,\mathscr C^\infty(\Omega)).
        \]
        Composing with the canonical isomorphism onto $\mathscr C^\infty(\Omega;\underline \Lambda^q)$ (as discussed in Section \ref{sec:cohomology_induced_subalgebra}), we obtain a map $B^q(\mathfrak u,\mathfrak t_\C,\mathscr C^\infty(G))^{\mathfrak t_\C} \to \mathscr C^\infty(\Omega;\underline \Lambda^q)$ which is easily verified to be the inverse of $\pi^{\ast}$. Therefore, the complexes are isomorphic and so are their cohomologies. 
    \end{proof}
            
    \begin{cor}
    Let $G, \mathfrak b$ and $\mathfrak t_\C$ be as in Proposition \ref{prp:cohomology_coeff_C}. Then we have 
    \[
    H^q(\mathfrak t_\C \oplus \mathfrak b, \mathfrak t_\C; \mathscr C^\infty(G)) =
    \begin{cases}
        \C &\textup{if $q =0$}, \\
        \{0\} &\textup{if $q>0$}
    \end{cases}
    \]
    \end{cor}
    \begin{proof}
         Let $\Omega = G/\T$ be equipped with the complex structure induced by $\pi_*(\mathfrak b)$. Then, following the results of \cite[Chapter 8]{besse_einstein_2008}, the manifold $\Omega$ has positive first Chern class. Thus, by \cite[Corollary~11.25]{besse_einstein_2008} we get
            \[
                H^{0,q}_{\dbar}(\Omega) = 
            \begin{cases}
                \C & \text{if } q = 0 \\
                \{0\}   & \text{if } q > 0.
            \end{cases}\]
    Hence, Proposition \ref{prp:cohomology_coeff_C} implies that $H^q(\mathfrak t_\C \oplus \mathfrak b, \mathfrak t_\C; \mathscr C^\infty(G))$ vanishes for $q>0$.
    \end{proof}
   
    \section{Fourier transform}
    \label{sec:fourier_transform}

    In this section, we briefly recall facts about left-invariant vector fields and Fourier analysis on compact Lie groups. For an introduction to Fourier analysis on compact Lie groups, we refer to \cite{ruzhansky_pseudo-differential_2010}. We also introduce some notation and results from \cite[Section~2]{doi_first_1987}, which we use in the proof of Theorem \ref{thm:main}.

    We identify $\mathfrak g$ with the set of left-invariant vector fields defined on $G$ by
    \[
        X f(x) \doteq \frac{\mathrm{d}}{\mathrm{d} t}\Big|_{t=0} f(x \exp(tX)),
    \]
    where $\exp$ is the exponential map from $\mathfrak{g}$ to $G$. Since $G$ is compact, we can assume that it is equipped with an $\Ad$-invariant Riemannian metric on $G$, which we denote by $\langle \cdot, \cdot \rangle$.

    We now want to set up the conditions to apply Theorem \ref{thm:hsc_thm_13} to the cohomology space $H^n(\mathfrak v, \mathscr C^\infty(G))$. To this purpose, we need to choose a family of finite-dimensional subspaces of the space $\mathscr C^\infty(G)$, which we obtain with the Fourier decomposition on $G$. We adopt the notation from \cite{doi_first_1987} for convenience, but also refer to \cite[Chapter 10]{ruzhansky_pseudo-differential_2010} for a detailed treatment with more modern notation. Note that the Riemannian metric $\langle \cdot , \cdot \rangle$ determines a unique bi-invariant Haar measure on $G$. We use the Haar measure to define the Laplace-Beltrami operator $\Delta_G$.

    Let $\widehat{G}$ denote the set of equivalence classes of irreducible, strongly continuous representations $\xi \colon G \to \End(H_\xi)$. Since $G$ is compact, each representation $\xi \in \widehat G$ acts on a finite-dimensional complex space $H_\xi$. The Fourier transform of a function $u \in \mathscr C^\infty(G)$ is defined by
    \[
        \widehat u(\xi)  = \int_G u(x) \xi(x)^\ast \mathrm{d}_G(x) \footnote{Here $\xi(x)^*$ denotes the conjugate transpose of $\xi(x)$.}, \qquad [\xi] \in \widehat{G}, 
    \] 
    where $\mathrm{d}_G$ is the Haar measure on $G$. The Fourier coefficient $\widehat u(\xi) $ defines an element in $\operatorname{End}H_\xi$.
    
    A smooth function $u$ can be recovered from its Fourier transform via
    \[
    u(x) = \sum_{[\xi] \in \widehat{G}} d_\xi \Tr\big(\widehat u(\xi)  \xi\big),
    \]
    where $d_\xi$ is the dimension of $H_\xi$ . The series converges in the $\mathscr C^\infty$-topology.
    
    The Fourier transform extends to Schwartz distributions on $G$, as well as to any of its subspaces, such as $L^2(G)$. For example, if $u \in L^2(G)$, the Fourier series converges to $u$ in the $L^2$-topology. Moreover, if $u \in L^2(G)$, then 
    \[
        \|u\|_{L^2}^2 = \sum_{[\xi] \in \widehat{G}}d_\xi \Tr\big(\widehat u(\xi) \widehat u(\xi)^\ast\big) = \sum_{[\xi] \in \widehat{G}}d_\xi \| \widehat u(\xi) \|_{\HS}^2.
    \]
    The norm $\| \cdot \|_{\HS}$ denotes the Hilbert--Schmidt norm on matrices. Further details can be found in \cite{ruzhansky_pseudo-differential_2010}.
    
    Since $G$ is a compact Lie group, each class $[\xi] \in \widehat{G}$ can be realized as a unitary representation of dimension $d_\xi$, with matrix elements $\xi_{i,j}$, $1 \leq i, j \leq d_\xi$. Because the metric $\langle \cdot, \cdot \rangle$ is $\Ad$-invariant, the operator $\Delta_G$ satisfies
    \[
        \Delta_G \xi_{i,j} = \Lambda_\xi \xi_{i,j}, \qquad i,j = 1, \dots, d_\xi,
    \]
    for some eigenvalue $\Lambda_{\xi} \in \R$ that depends only on the class $[\xi]$. We write $\langle \xi \rangle = (1+\Lambda_{[\xi]})^{1/2}$.

    For each class $[\xi] \in \widehat G$, let $H_{\xi} = \operatorname{span}_\C \{ \xi_{i,j} \, | \, 0 \leq i, j, \leq d_\xi \}.$ We can write
    \[ \begin{aligned}
        \mathscr C^\infty(G) & \cong \widehat{ \displaystyle \bigoplus_{[\xi] \in \widehat G} } \End(H_\xi) \\
        & \doteq \left\{(\widehat u(\xi))_{[\xi]\in \widehat{G}} \in \prod_{[\xi] \in \widehat G}\End(H_\xi) \, \middle| \,\sum_{[\xi] \in \widehat G} d_\xi \Tr(\widehat u(\xi) \xi) \text{ converges in }\mathscr C^\infty(G) \right\},
    \end{aligned}
    \]
    where the right side is understood as the set of sequences in $[\xi] \in \widehat G$ for which the corresponding Fourier series converges in the $\mathscr C^\infty$-topology. That is, we only consider sequences $\{\widehat a(\xi)\}_{[\xi] \in \widehat G}$, with $\widehat a(\xi) \in \End(H_\xi)$, such that
    \begin{equation}
    \label{eq:smooth_ineq_fourier}
        \sum_{[\xi] \in \widehat G} d_\xi \langle \xi \rangle^{k} \|\widehat a(\xi)\|_{\HS} < \infty,\,\,\,k\in \Z_+.
    \end{equation}

    Since $\Delta_G$ commutes with left-invariant differential operators we have that $XH_{\xi} \subset H_{\xi}$ for all $X \in \mathfrak g_\C$. Therefore, the symbol of a left-invariant vector field, which we denote by $\widehat L$, satisfies $\widehat L (\xi) \in \End(H_\xi)$ for all $[\xi] \in \widehat G$. Thus $\End(H_\xi)$ is a $\mathfrak g_\C$-module for each $[\xi]$. They are also $\mathfrak v$-modules for any subalgebra $\mathfrak v \subset \mathfrak g_\C$.

    For each $q =1, 2, \ldots, n$, we identify $\mathscr C^\infty(G, \underline \Lambda^{q})$ with the tensor product of vector spaces $\mathscr C^\infty(G) \otimes \underline \Lambda^{q}$. Since $\underline \Lambda^{q}$ is finite-dimensional, the universal property of the tensor product gives us a unique extension of the Fourier transform from $\mathscr C^\infty(G)$ to $\mathscr C^\infty(G, \underline \Lambda^{q})$, which is independent of choice of basis for $\underline \Lambda^{q}$.

    Therefore, for each $[\xi] \in \widehat G$, if $[u] \in H^q(\mathfrak v, \mathscr C^\infty(G)) \cong H^q(G, \mathfrak v)$ (see Proposition~\ref{prp:iso_dbar_algebra}) and $v \in [u]$, then there exists an element $w \in \mathscr C^\infty(G, \underline \Lambda^{q-1})$ such that
    \[
        u - v = \dbarb w.
    \]
    It follows that the Fourier coefficients satisfy
    \[
        \widehat u(\xi) - \widehat v(\xi) = \widehat{\dbar}_b(\xi) \widehat w(\xi),
    \]
    where $\widehat \dbar_b(\xi)\colon \End(H_\xi)\otimes \underline \Lambda^q \to \End(H_\xi)\otimes \underline \Lambda^{q+1}$ denotes the induced operator on each Fourier level. 
    
    The map $[u] \mapsto [\widehat u(\xi)]$ defines a decomposition
    \begin{align}\label{eq:iso_fourier_dec}
        \begin{split}
        H^q(\mathfrak v & ,  \mathscr C^\infty(G))    \cong \widehat{ \displaystyle \bigoplus_{[\xi] \in \widehat G} } H^q(\mathfrak v, \End(H_\xi))  \\
        &\doteq\left\{([ \widehat u(\xi)]) \in \prod_{[\xi]\in \widehat G}H^q(\mathfrak v, \End(H_\xi))\, \middle| \, \sum_{[\xi]\in \widehat{G}} d_\xi \Tr(\widehat u(\xi) \xi ) \text{ converges}
        \right\},
        \end{split}
    \end{align}
    where the hat over the direct sum indicates that on the right side, we have sequences converging in the $\mathscr C^\infty$-topology induced in $H^q(\mathfrak v, \mathscr C^\infty(G))$. We extend this decomposition to functions in $\mathscr C^\infty(\Omega)$, by identifying them with the subspace of $\T$-periodic functions $\mathscr C^\infty(G)^\T \subset \mathscr C^\infty(G)$.

    \begin{prp} \label{prp:really_important_prp}
        With $G$, $\mathfrak t_\C$, and $\mathfrak b$ as in Proposition \ref{prp:cohomology_coeff_C}, for $q = 0,1,\ldots,l$ we have
        \[
            H^q(\mathfrak t_\C \oplus \mathfrak b, \mathfrak t_\C; \operatorname{End}(H_\xi)) = 
            \begin{cases}
                \C & \text{if } q = 0 \text{ and } \xi \text{ trivial,} \\
                \{0\}   & \text{ otherwise. }
            \end{cases}\]
    \end{prp}
    \begin{proof}
        Using the Fourier transform, we can write a decomposition of the form
        \[ 
             H^q(\mathfrak t_\C \oplus \mathfrak b, \mathfrak t_\C; \mathscr C^\infty(G)) \cong  \widehat{ \displaystyle{ \bigoplus_{[\xi] \in \widehat G}} }H^q(\mathfrak t_\C \oplus \mathfrak b, \mathfrak t_\C; \operatorname{End}(H_\xi)). 
             \]
            By Proposition \ref{prp:cohomology_coeff_C},
            the left-hand-side is $\C$ when $q=0$ and it vanishes otherwise. It follows that $H^q(\mathfrak t_\C \oplus \mathfrak b, \mathfrak t_\C; \operatorname{End}(H_\xi)) = 0$ for $q > 0$ and all $\xi \in \widehat G$. Furthermore, since $H^0(\mathfrak t_\C \oplus \mathfrak b, \mathfrak t_\C; \operatorname{End}(H_0)) = \C$, we also have that $H^0(\mathfrak t_\C \oplus \mathfrak b, \mathfrak t_\C; \operatorname{End}(H_\xi)) = \{0\}$ for all non-trivial $\xi \in \widehat G.$
    \end{proof}

    \section{Semilocal structure}
    \label{sec:semilocal}

    Let $\mathfrak v$ be a left-invariant CR structure on $G$ with toral part $\mathfrak m$ and associated maximal torus $\T$. The goal of this section is to prove the following proposition, which will be used in Section~\ref{sec:the_proof} to show that certain cohomology spaces vanish in the Hochschild--Serre decomposition.
     \begin{prp}
    \label{prp:an_inclusion}
        If $\mathfrak m$ satisfies \DC and $\xi \in \widehat G$, then \[
            H^q(\mathfrak m \oplus \mathfrak b, \mathfrak m; \End(H_\xi)) \subset H^q(\mathfrak t_\C \oplus \mathfrak b, \mathfrak t_\C; \End(H_\xi)).
        \]
    \end{prp}

    The proof of Proposition \ref{prp:an_inclusion} relies on a semilocal analysis of the CR structure in a tubular neighborhood of the maximal torus $\T$.

    Let $\Omega = G / \T$ and observe that the pushforward $\pi_\ast (\mathfrak v)$ defines a complex structure on $\Omega$. By \cite[Corollary 10.1.11]{hilgert_structure_2012}, there exists a covering $\mathfrak U$ of $\Omega$ by open sets such that for each $U \in \mathfrak U$, there is a smooth section $\sigma = \sigma_U$ of $\pi \colon G \to \Omega$ such that the map
    \begin{equation}
    \label{eq:hilgert_diffeo}
        \Psi\colon U \times \T \to  \sigma(U) \T, \qquad (z, t) \mapsto \sigma(z)t
    \end{equation} is a diffeomorphism onto an open subset $V = \sigma(U) \T$ of $G$. We can further shrink $U$ such that it has complex coordinates $(z_1, \ldots, z_l)$. For each $U \in \mathfrak U$, let $V = \pi^{-1}(U)$. We define functions $\overline{Z}_j(g) \doteq \overline z_j \circ \pi(g)$ for $j = 1, \ldots, l = (N - d)/2$.

    Let $m\doteq \dim \mathfrak m$, let $\overline{T}_1, \ldots, \overline{T}_m$ be a basis for $\mathfrak m$. Write $\overline{T}_j=\frac{1}{2}(X_j+iY_j)$, where $X_j,Y_j$ are real. Since $\mathfrak m \cap \overline{\mathfrak m}=\{0\}$, the vector fields $\{X_j,Y_k\}$ are linearly independent. Let $\{\eta_1,\ldots,\eta_m,\rho_1,\ldots,\rho_m\}$ be the dual basis, which we can extend as real $1$-forms in $G$. Then, if we set $\tau_j \doteq \eta_j + i\rho_j$, we have $\tau_j(T_k)=\delta_{jk}$ and $\tau_j(\overline{T}_k)=0$ for all $j,k=1,\ldots,m$. Let $\xi$ be any non-zero real 1-form in $\Ann \mathfrak v \cap \Ann \overline{\mathfrak v} \subset \mathfrak g_\C$ (we take it to be the dual of the missing direction $X$). Then, $\Ann \mathfrak v|_V$ is spanned by $$\tau_1, \ldots, \tau_m, \xi, \dext Z_1, \ldots, \dext Z_{l}.$$ 
    Moreover, the entire $T_\C^* V$ is spanned by
    $$\tau_1, \ldots, \tau_m, \overline \tau_1, \ldots, \overline \tau_m, \xi, \dext Z_1, \ldots, \dext Z_{l},  \dext \overline Z_1, \ldots, \dext \overline Z_{l}.$$ 

    Using the (real) diffeomorphism $\Psi$, we pull these forms back to $U \times \T$. This yields the following forms:
    \[\tau_j' = \Psi^*(\tau_j), \quad \overline \tau_j' = \Psi^*(\overline \tau_j), \qquad \xi' = \Psi^* \xi, \quad Z_k' = Z_k \circ \Psi,  \quad \overline Z_k' = \overline Z_k, \quad X'=\Psi^\ast{X}\] for $j = 1, \ldots, m$ and $k = 1, \ldots, l$.

    For all $j,j' = 1, \ldots, m$ and $k,k' = 1, \ldots, l$, the following hold:
    \[ \tau'_j(0|_z \oplus \overline T_{j'}|_t) =  \tau_j(\Psi_*( 0|_z \oplus \overline T_{j'}|_t)) =  \tau_j(\overline T_{j'}|_{\sigma(z)t}) = 0,\]
    \[ \xi'(0|_z \oplus \overline T_{j}|_t) =  \xi(\Psi_*( 0|_z \oplus \overline T_j|_t)) =  \xi(\overline T_j|_{\sigma(z)t}) = 0,\]    
    
    \[ \dext Z'_k(0|_z \oplus \overline T_j|_t) = \dext Z_k(\Psi_*( 0_z \oplus \overline T_j|_t)) = \dext \overline Z_k(T_j|_{\sigma(z)t}) = 0,\]
    and
    \[ \begin{aligned}
        \dext Z'_k(\del/\del \overline z_{k'} |_z \oplus 0_t) & = \dext Z_k ( \Psi_* ( \del/\del \overline z_{k'} |_z \oplus 0_t)) \\
        & = \dext Z_k(R_t)_*(\sigma_*(\del/\del \overline z_{k'}|_z )) \\
        & = \dext z_k (\pi \circ R_t \circ \sigma )_*(\del / \del \overline z_{k'}) \\
        & = \dext z_k (\del / \del \overline z_{k'}) = 0.
    \end{aligned} \]

    We define $$f_{jk} \doteq \tau'_j(\del / \del \overline z_k \oplus 0) = \tau_j(\Psi_*(\del / \del \overline z_k \oplus 0)) \in \mathscr C^\infty(U \times \T), $$
    
    $$ h_{k} \doteq \xi'(\del / \del \overline z_k \oplus 0) = \xi(\Psi_*(\del / \del \overline z_k \oplus 0)) \in \mathscr C^\infty(U \times \T), $$
    $$ \overline {T_j'} \doteq 0 \oplus \overline T_j$$ for $j = 1, \ldots, m$, $k=1,\ldots,l$ and $$\overline L_{k} \doteq \del / \del \overline z_k \oplus 0 - \sum_{j = 1}^m 
    f_{jk} (0 \oplus T_j) 
    - h_k X',\,\,k=1,\ldots,l.$$ For readability, we will omit the symbol $\oplus$ in the following.

    We conclude that the CR structure on $U \times \T$ given by
    \[
        \mathcal W_{(z,t)} = \operatorname{span}_\C \{ \overline{T'_j}, \overline{L}_k|_{(z,t)} \, |\, j = 1, \ldots, m, ~k =1, \ldots, l\}
    \]
    and $\mathcal W = \bigcup_{(z,t) \in U \times \T} \mathcal W_{(z,t)}$ makes the map $\Psi$ a CR map. Moreover, the functions $f_{jk}$  and $h_{k}$ are constant over $\T$.

    Now we have that
    \begin{equation} \label{eq:local_frame}
        T_1', \ldots, T_m', \overline {T_1'}, \ldots, \overline {T_m'}, X', L_1, \ldots, L_l, \overline L_1, \ldots, \overline L_l
    \end{equation}
    form a basis for $T_\C (U \times \T)$. Now we need a good representation for $\dbarb$ on $U \times \T$, for which it suffices to find a convenient dual basis for \eqref{eq:local_frame}. 

    We define $\overline w_j = \overline \tau_j - \sum_{k=1}^l f_{jk}\dext Z_k$ and observe that
    $\overline w_j (T_k) = \overline w_j (L_k) = \overline w_j (\overline L_k) = 0$ and 
    $\overline w_j (\overline T_k) = \delta_{jk}.$ It is easy to verify that the following forms are dual to the vector fields \eqref{eq:local_frame}:
    \[\tau_1', \ldots, \tau_m', \overline \omega_1, \ldots, \overline \omega_m, \xi, \dext Z_1, \ldots, \dext Z_l, \dext \overline Z_1, \ldots, \dext \overline Z_l. \]

    \begin{lem}
    \label{lem:invariance_functions}
        If $U \subset G/\T$ is an open neighborhood with complex coordinates $z=(z_1, \ldots, z_l)$ and $f \in \mathscr C ^\infty(\pi^{-1}(U))$ is such that $\dbarb f = 0$ over $\pi^{-1}(U)$, then $f$ is constant along $\pi^{-1}(z)$ for all $z \in U$, and thus descends to a holomorphic function on $U$. 
    \end{lem}

    \begin{proof}
        Notice that, since $\overline \tau_j$ and $\dext \overline Z_k$ are $\dbarb$-closed, and by writing $f' = f \circ \Psi^{-1} $, we have
        $$ 0 = \dbarb f' = \sum_{j=1}^m \left(\overline{T}_j f'\right) \overline{\omega}_j + \sum_{k=1}^{l} \left(\overline{L}_k f'\right) \dext \overline{Z}_k $$
        and so
        $ 0 = \dbarb f' = \sum_{j=1}^m \left(\overline{T}_j f' \right) \overline \omega_j$. Now we take the Fourier transform with respect to the variable $t$. In the following, we denote by $\dbarb'$ the tangential Cauchy--Riemann operator with respect to the structure defined by $\mathfrak m \oplus \{0\}$, and we have
        $$ \begin{aligned}
            \widehat{\dbarb' f'}(\xi, z) 
            & = \int_{\mathbb T}  \dbarb'(f')(t,z) e^{-i \xi t } \dint \mu(t) \\
            & = \sum_{j=1}^m \widehat{ \overline {T}}'_j(\xi) \widehat f'(\xi, z)  \overline \tau_j.
        \end{aligned} $$
        
        Now \DC implies that, for each $\xi\not=0$, at least one of the $\widehat {\overline {T}}'_j(\xi)$ is non-zero, which implies that $\widehat f'(\xi, z)$ is zero for $\xi\not=0$, and thus $f'$ is constant in the variable~$t$.
    \end{proof}

    \begin{proof}[Proof of Proposition \ref{prp:an_inclusion}]
        Let $[\widehat u(\xi)] \in H^q(\mathfrak m \oplus \mathfrak b, \mathfrak m; \End(H_\xi))$. Consider the smooth $q$-form $u \doteq \operatorname{Tr}( \hat u(\xi) \xi)$. It suffices to show that $\mathcal L_Z u = 0$ for all $Z \in \mathfrak m$ implies $\mathcal L_Z u = 0$ for all $Z \in \mathfrak t_\C$. Note that $\widehat u(\xi)$ is the Fourier coefficient of $u$.
        
        Since the Lie derivative is local, it suffices to verify this in open neighborhoods. Let $U \subset G/\T$ be an open neighborhood with complex coordinates $z=(z_1, \ldots, z_l)$ and let $V = \pi^{-1}(U)$. We can write $u|_V$ as
        \[
            u = \psum_{|I|=q} u_I  \dext \overline Z_I,
        \]
        where the $'$ symbol on the sum indicates that the sum is over increasing multi-indices, and since each $\dext \overline Z_I$ is $\T$-invariant, we have that $\mathcal L_Z \dext \overline Z_I = 0$ for all $I$, and thus
        \[
           0 = \mathcal L_Z u = \psum_{|I|=q} (\mathcal L_Z u_I) \dext \overline Z_I + u_I \mathcal L_Z(\dext \overline Z_I) = \psum_{|I|=q} \mathcal L_Z u_I \dext \overline Z_I
        \]
        which implies that $\mathcal L_Z u_I = Z u_I = 0$. By \DC and Lemma \ref{lem:invariance_functions}, $u_I$ is constant along $\T$, so $\mathcal L_Z u_I = 0$ for all $Z \in \mathfrak t$ and by linearity for all $Z \in \mathfrak t_\C$. Since $U \subset G/\T$ was arbitrary, this completes the proof.
    \end{proof}

    \section{Proof of Theorem \ref{thm:main}}
    \label{sec:the_proof}
    Let $\mathfrak v$ be a subalgebra of $\mathfrak g_\C$ corresponding to a left-invariant CR structure on the compact Lie group $G$. As explained in Section \ref{subsect: CR structures on hypersurface type}, we have a decomposition $\mathfrak v = \mathfrak m \oplus \mathfrak b$, where $\mathfrak m \subset \mathfrak t_\C$ is a bi-invariant CR structure on the maximal torus $\T \subset G$. Suppose the toral part $\mathfrak m$ satisfies \DC (see Section \ref{sec:division_condition}).  Our goal is to prove that the inclusion
    \begin{equation}
        H^q(\T, \mathfrak m) \hookrightarrow H^q(G, \mathfrak v)
    \end{equation}
    constructed in Section \ref{sec:involutive_structures} is an isomorphism. Combining the classification Theorem \ref{thm:classification_theorem} and Theorem \ref{thm:nodc_nocr1} with the \DC condition, we get that $\mathfrak v$ is of type $\CRO$ and thus $\mathfrak b \subset \mathfrak v$ is an ideal. Now, by Proposition \ref{prp:iso_dbar_algebra}, the cohomology space $H^{q}(G,\mathfrak v)$ is isomorphic to $H^q(\mathfrak v, \mathscr C^\infty(G))$, to which we apply the Fourier decomposition constructed in \eqref{eq:iso_fourier_dec}, obtaining
        \[
            H^q(\mathfrak v, \mathscr C^\infty(G)) \cong \widehat \bigoplus_{[\xi] \in \widehat{G}}H^q(\mathfrak v, \End(H_\xi)).
        \]
    From here, it remains to prove the following three claims:
    \begin{enumerate}[label=(\roman*)]
            \item \label{enu:trivial} If $\xi$ is the trivial representation, then $H^q(\mathfrak v, \End(H_\xi)) = H^q(\T, \mathfrak m)$.
            \item \label{enu:non_trivial} If $[\xi] \in \widehat G$ is non-trivial, then $H^q(\mathfrak v, \End(H_\xi)) = \{0\}$. In particular, for all $[u] \in H^q(\mathfrak v, \mathscr C^\infty(G))$ there is a sequence $[\widehat v(\xi)] \in H^q(\mathfrak v, \End(H_\xi))$, such that for all non-trivial $\xi$ we have $\widehat u (\xi) = \widehat \dbar_b \widehat v(\xi)$.
            \item \label{enu:convergence} The series \[ v~\doteq~\sum_{[\xi] \in \widehat G, ~ \xi \neq 0} d_\xi \operatorname{Tr}(\widehat v(\xi)\xi)\] is convergent. In particular, we have $u - \widehat u(0) = \dbarb v$, where $0$ denotes the trivial representation. Thus, we get $[u] = [\widehat u(0)] \in H^q(\mathfrak v, \End(H_0))$.
    \end{enumerate}
    Combining these three items we get that every cohomology class in $H^q(\mathfrak v, \mathscr C^\infty(G))$ has a representative in $H^q(\mathfrak v, \End(H_0)) = H^q(\T, \mathfrak m)$. Therefore, the injective map $H^q(\T, \mathfrak m) \to H^q(\mathfrak v, \mathscr C^\infty(G)) \cong H^q(G, \mathfrak v)$ is also surjective and thus an isomorphism.

    We begin with the first two items. First, we set up the necessary conditions to apply Theorem \ref{thm:hsc_thm_13} to the cohomology spaces $H^q(\mathfrak v, \End(H_\xi))$. Consider the decomposition $\mathfrak v = \mathfrak m \oplus \mathfrak b$. Since $\mathfrak m$ is abelian it is reductive in $\mathfrak v$, so it remains to verify that $\End(H_\xi)$ is completely reducible as an $\mathfrak m$-module. Indeed, note that the finite-dimensional representation $\mathfrak m \to \End(H_\xi)$ extends naturally to $\rho_{\xi} \colon \mathfrak g_{\C}\to \End(H_\xi)$, by differentiation with respect to left-invariant vector fields. Moreover, since $\mathfrak t_\C$ is a Cartan subalgebra of the Lie algebra $\mathfrak g_\C$, it follows that $\rho_\xi(\mathfrak t_\C)$ acts diagonally on $\End(H_\xi)$ (see, for example, \cite[Lemma 4.11.3]{duistermaat_lie_2000}.) In particular, $\mathfrak m \subset \mathfrak t_\C$ acts diagonally, and since $\mathfrak m$ is abelian, all its elements are simultaneously diagonalizable, which implies that $\End(H_\xi)$ is a completely reducible $\mathfrak m$-module (see also \cite[Theorem 5.1]{Serre1992}). Therefore, Theorem \ref{thm:hsc_thm_13} yields the isomorphism
        \begin{equation}\label{eq:cohomology}
            H^q(\mathfrak v, \End(H_\xi)) \cong \sum_{j + k = q}  H^j(\mathfrak m, \C) \otimes H^k(\mathfrak b, \End(H_\xi))^\mathfrak{m},\,\,q \geq 0,\,[\xi] \in \widehat{G}.
        \end{equation}

        From the \DC condition, we have that $H^j(\mathfrak m, \C) = H^j(\T, \mathfrak m)$. This follows from the fact that each cohomology class of the latter has a representative in the former \cite[Lemma 7]{jacobowitz_levi-flat_2023}. The following lemma, a special case of \cite[Theorem 2.28]{solleveld}, allows us to identify $H^k(\mathfrak b, \End(H_\xi))^{\mathfrak m}$ with the relative cohomology from Proposition~\ref{prp:really_important_prp}.

        \begin{lem} \label{lem:relative_cohomology_identification}
            Let $\mathfrak u = \mathfrak a \oplus \mathfrak b$, with $\mathfrak a \in \{\mathfrak m,  \mathfrak t_\C\}$, and let $F = \End(H_\xi)$ for some representation $\xi$. The inclusion $\mathfrak b \hookrightarrow \mathfrak u$ induces an isomorphism
            \[
                H^\bullet(\mathfrak u, \mathfrak a; F) \to H^\bullet(\mathfrak b, F)^{\mathfrak a}.
            \]
        \end{lem}

        To prove \ref{enu:trivial}, we combine \eqref{eq:cohomology} and Proposition \ref{prp:really_important_prp} with the previous lemma to obtain
        \[ H^q(\mathfrak v, \End(H_0)) =  H^q(\T, \mathfrak m).\]
        The second item follows directly from
        \[
        H^k(\mathfrak b, \End(H_\xi))^\mathfrak{m} = H^k(\mathfrak m \oplus \mathfrak b, \mathfrak m;\End(H_\xi)) \subset H^k(\mathfrak t_\C \oplus \mathfrak b, \mathfrak t_\C;\End(H_\xi)) = \{0\}.
        \]
        Here, the first equality is due to Lemma \ref{lem:relative_cohomology_identification}, the inclusion is given by Proposition~\ref{prp:an_inclusion}, and the final equality follows from Proposition \ref{prp:really_important_prp}.

    It remains to prove \ref{enu:convergence}. Consider the map
    \begin{align}
        \Phi \colon H^q(G, \mathfrak v) \to H^q(\mathfrak v, \C),
    \end{align}
    defined in the following way: given $[u]\in H^q(G, \mathfrak v)$, we write the Fourier expansion
    \[
    u=\sum_{[\xi] \in \widehat{G}}d_\xi \Tr(\widehat{u}(\xi)\xi),
    \]
    and let 
    \[
    \Phi([u]) \doteq [\widehat{u}(0)] \in H^q(\mathfrak v, \C).
    \]
    which is, concretely, taking the average of $u$ along the Haar measure $\mathrm{d}_G$ and viewing the resulting (closed) left-invariant differential form as an element of $H^q(\mathfrak v, \C)$. The map $\Phi$ is a well-defined, surjective linear map. Our goal is to prove that $\Phi$ is injective.
    
    We must therefore prove the following: given a cohomology class $[u]\in H^q(G, \mathfrak v)$, we write the decomposition 
    \[
        u = \sum_{[\xi] \in \widehat{G}} d_\xi \operatorname{Tr}(\widehat{u}(\xi)\xi),
    \]
    where $\widehat{u}(\xi)\in \End(H_\xi)\otimes \underline \Lambda^q$. Suppose that $\widehat{u}(0)$ is exact, then we have to show that $u$ is exact.

    From the argument immediately following Lemma \ref{lem:relative_cohomology_identification}, $\widehat u(\xi)$ is exact for every $[\xi] \in \widehat{G}$. For each $q =1, 2, \ldots, n$, we identify $L^2(G, \underline \Lambda^{q})$ with $L^2(G) \otimes \underline \Lambda^{q}$ and thus we extend the $L^2$ inner product from $L^2(G)$ to $L^2(G, \underline \Lambda^{q})$. The definition of the norm in $L^2(G, \underline \Lambda^{q})$ depends on the inner product of $\mathfrak g_\C^*$, which defines an inner product on $\underline \Lambda^{q}$. Notice that, if we chose a different inner product on $\mathfrak g_\C^*$, we would have a different inner product on $L^2(G) \otimes \underline \Lambda^{q}$. However, all such norms on $L^2(G, \underline \Lambda^{q})$ are equivalent and thus define the same $L^2$-topology.

    For the remainder of the proof, it is advantageous to work on the quotient $\Omega = G/ \T$ rather than $G$ itself. This is because the operator $\dbar$ on $\Omega$ is elliptic, and therefore we can apply Hörmander $L^2$-estimates to control the growth of certain Fourier coefficients. We denote by $(H_\xi^*)^\T$ the set of elements in $H_\xi^*$ that are invariant under the action of the torus $\T$. In view of \cite[Lemma 1, Lemma 2]{doi_first_1987}, and using the identification $\End(H_\xi) \cong H_\xi \otimes (H_\xi^*)^\T$ we have
    \[
        L^2 (\Omega) = \bigoplus_{[\xi] \in \widehat G} H_\xi \otimes (H_\xi^*)^\T.
    \]
    Now, since for every $\xi$, there is $\widehat v(\xi) \in C^{q-1}(\mathfrak v, \End(H_\xi))$ solving $\widehat \dbar_b(\xi) \widehat v(\xi) = \widehat u(\xi)$, we can define $$f \doteq \Tr(\widehat v(\xi) \xi) \in \mathscr C^\infty(\Omega, \Lambda^{q-1}),$$ and $$g \doteq \Tr(\widehat u(\xi) \xi) \in \mathscr C^\infty(\Omega, \Lambda^q)$$ and these forms satisfy
    \[
        \dbarb (\Tr(\widehat v(\xi) \xi)) = \Tr(\widehat u(\xi) \xi).
    \]

    We now explain how to estimate $\|f\|_{L^2}$ using $\|g\|_{L^2}$. Consider the linear, densely defined operator
    \[
        \dbar_q \colon L^2(\Omega, \underline \Lambda^{q}) \to L^2(\Omega, \underline \Lambda^{q+1}).
    \]
    The adjoint of $\dbar$ is also linear and densely defined and is denoted by
    \[
        \dbar^*_q \colon L^2(\Omega, \underline \Lambda^{q+1}) \to L^2(\Omega, \underline \Lambda^{q}).
    \]
    The associated box operator is defined by $\Box_q = \dbar^*_q \dbar_q + \dbar_{q-1} \dbar^*_{q-1}$ and we have that
    $H^{0,q}_{\dbar}(\Omega, \C) \cong \ker \Box_q$. 

    If $f \in \ker \Box_q$, then
    \[
        0 = \langle \Box_q f, f \rangle = \langle \dbar f, \dbar f \rangle + \langle \dbar^* f, \dbar^* f \rangle.
    \]
    This implies that $\dbar f = 0$ and $\dbar^* f = 0$.
    
    Moreover, $\dim H^{q}_{\dbar}(\Omega, \C) < \infty$ and thus $\dbar$, as an operator densely defined on $L^2(\Omega, \underline \Lambda^{q})$ with range in $L^2(\Omega, \underline \Lambda^{q+1})$, has closed range.

    By \cite[Theorem 1.1.1]{hormander_l2_1965}, there exists a constant $C > 0$ such that for all $f \in \ker \Box_q^\perp$ in the range of $\dbar_{q}$, if $u \in \operatorname{dom}(\dbar_{q-1}) \cap \operatorname{ran}(\dbar^*_{q-1})$ is such that $\dbar u = f$, 
    then
    \[
        \|u\|_{L^2} \leq C\|f\|_{L^2}.
    \]
    We can therefore find, for every $[\xi] \in \widehat{G}$, a $(q-1)$-form $\widehat v(\xi)$ with values in $H_\xi$ solving $\dbar v(\xi) = u(\xi)$ satisfying
    \[
        \|\widehat v(\xi)\|_{\HS} \leq C \|\widehat u(\xi)\|_{\HS},
    \]
    where the constant is independent of $\xi$. Since $u$ is smooth, by using the previous inequality and \eqref{eq:smooth_ineq_fourier}, we conclude that the form $v \doteq \sum_{[\xi] \in \widehat G} d_\xi \operatorname{Tr}(\widehat v(\xi)\xi)$ is smooth and satisfies $\dbarb v = u$, as we wanted to show.
 
    \section{Proof of Theorem \ref{thm:necessity}}
    \label{sec:necessity_proof}
    In order to complete the proof of Theorem \ref{thm:necessity}, we need to establish the following result:
    \begin{thm}\label{thm:cr0_infinite} Let $\mathfrak{v}=\mathfrak{m}\oplus \mathfrak b_{\mathfrak t}$ be a $\CRO$ structure on the compact Lie group $G$. If $\mathfrak m$ does not satisfy \DC, then 
    \[
    H^\bullet(G,\mathfrak v)=\bigoplus_{q\geq 0 }H^q(G,\mathfrak v)
    \]
    is infinite-dimensional.
    \end{thm}
    This requires some preliminaries on weights and weight representations, which we discuss next.
    \subsection{Preliminaries on weights}
    We retain the notation of the previous sections. Given the maximal torus $\T$, we let $X^\ast(\T)\doteq\Hom(\T,S^1)$ be the character lattice of $\T$, whose elements we refer to as \textit{integral weights}. Once coordinates $\T^d\simeq \R^d/\Z^d$ are chosen, one can identify $X^\ast(\T)$ with $\Z^d$ via the Fourier frequencies. We shall not distinguish notationally between a character $\chi \in X^\ast(\T)$ and its differential, viewed as an element of $\mathfrak t^\ast_\C$. Thus, expressions like $\chi(H)$ for $H\in \mathfrak t_\C$ refer to the differentiated character.  We use the notation $W=N_G(\T)/\T$ for the Weyl group, which acts on $X^\ast(\T)$ by pullback.  

    Recall that, given a root $\alpha \in \Delta$, its \textit{coroot} $\alpha^{\vee}$ is defined as the unique element in $[\mathfrak g_\alpha, \mathfrak g_{-\alpha}]\subset \mathfrak t_\C$ which satisfies $\alpha(\alpha^{\vee})=2$. The \textit{dominant chamber} is then the closed cone
    \[
    \mathcal{C}^{+}=\{\lambda \in \mathfrak t^{\ast};\,\langle \lambda,\alpha^{\vee}_i\rangle \geq 0\text{ for all simple roots }\alpha_i\}.
    \]
    If an integral weight belongs to $\mathcal{C}^{+}$, we say that it is \textit{dominant}. Moreover, given any integral weight $\mu \in X^{\ast}(\T)$, there is an element $w \in W$ such that $w\mu$ is dominant.

    Given a finite-dimensional complex representation $(\pi,V)$ of $G$, its restriction to $\T$ is completely reducible, and hence decomposes into one-dimensional characters
    \[
    V=\bigoplus_{\mu \in X^\ast(\T)}V_\mu,
    \]
    where $V_\mu\doteq \{v\in V;\,\pi(t)v=\mu(t)v,\,\,\forall t\in \T\}$.
    The $\mu \in X^\ast(\T)$ such that $V_\mu\not=\{0\}$ are the \textit{weights} of the representation $\pi$, and the spaces $V_\mu$ are the corresponding weight spaces. Note that, if $X\in \mathfrak g_\alpha$, then
    \[
    \mathrm{d}\pi(X)V_\mu \subset V_{\mu+\alpha},
    \]
    so the root vector $X_\alpha$ raises the weight by $\alpha$. A non-zero vector $v\in V$ is called a \textit{highest-weight vector} if there exists a weight $\lambda \in X^\ast(\T)$ such that $\mathrm{d} \pi(H)v = \lambda(H)v$ for all $H\in \mathfrak t_\C$ and $\mathrm{d} \pi(X)v=0$ if $X\in \mathfrak b_{\mathfrak t}$. The following theorem will be crucial in what follows, see for instance \cite[Theorem 4.9.5]{duistermaat_lie_2000}:
    
    \begin{thm}[Highest-weight classification]
    \label{thm:highest-weight-classification}
    For every dominant integral weight $\lambda\in X^*(\T)\cap \mathcal{C}^+$, there exists
    a unique irreducible finite-dimensional representation $(\pi_\lambda,V_\lambda)$ of
    $G$ whose highest weight is $\lambda$. Conversely, every irreducible
    finite-dimensional representation of $G$ is of this form.
    \end{thm}
    Now we can move on to the proof of Theorem \ref{thm:cr0_infinite} (and consequently, of Theorem \ref{thm:necessity}).
    \subsection{Proof of Theorem \ref{thm:cr0_infinite}}

    We write $\mathfrak v= \mathfrak m \oplus \mathfrak b_{\mathfrak t}$, and assume that $\mathfrak m$ does not satisfy \DC. Let $L_1,\ldots,L_r$ be a basis of $\mathfrak m$. 
    \begin{lem}\label{lem:dominant_bad_weights} There exists a sequence of pairwise distinct dominant integral weights $\lambda_n$ and an element $g\in G$ such that the basis of $\Ad(g)\mathfrak m$ induced by $L_1,\ldots,L_r$, which we denote by $\widetilde{L_1},\ldots,\widetilde{L_r}$, satisfies
    \[
    \varepsilon_n\doteq \max_{1\leq j \leq r}\left|\lambda_n(\widetilde{L_j})\right| \leq (1+|\lambda_n|)^{-2n},\qquad n\in \N.
    \]
    \end{lem}
    \begin{proof} Failure of \DC means that there is a sequence $\mu_n$ of integral weights such that 
    \[
    \max_{1\leq j \leq r}|\mu_n(L_j)|\leq (1+|\mu_n|)^{-2n},\,\,n\in \N.
    \]
    For each $n$, there is an element $w_n\in W$ such that $w_n\mu_n$ is dominant. Since $W$ is a finite group we can, after passing to a subsequence, find a single element $w$ such that $w\mu_n$ is dominant for every $n$. Choosing a representative $g\in N_G(\T)$ for $w\in W$, we obtain the result.
    \end{proof}
    Since replacing $\mathfrak v$ with $\Ad(g)\mathfrak v$ does not change any of the properties of the CR complex, we shall drop the tildes on $\widetilde{L}_j$.

    We shall use the right action of $\T$ on the CR complex. Since $\mathfrak v$ is $\Ad(\T)$-invariant, right translation by elements of $\T$ acts on $C^q(\mathfrak v, \mathscr C^\infty(G))$, in the following way: given $u\in C^q(\mathfrak v, \mathscr C^\infty(G))$ and $t\in \T$,
    \[
    (R^{\ast}_tu)_g(X_1\ldots,X_q)=u_{gt}(\Ad(t^{-1})X_1,\ldots,\Ad(t^{-1})X_q),\qquad X_j\in \mathfrak v.
    \]
    The operator $\dbar_b$ commutes with this action, hence the CR complex decomposes into $\T$-isotypic components.

    Let $X_\alpha\in\mathfrak g_\alpha$, $\alpha\in\Delta_+$, and let
    $\omega^\alpha$ denote the corresponding dual covector in $\mathfrak v^*$.
    Let $\tau^1,\ldots,\tau^r$ be a basis of $\mathfrak m^*$. The forms
    $\tau^j$ have right $\T$-weight $0$, while $\omega^\alpha$ has right
    $\T$-weight $-\alpha$. Indeed,
    \[
    \Ad(t)M=M,\qquad M\in\mathfrak m,
    \]
    because $\mathfrak m\subset\mathfrak t_\C$ and $\T$ is abelian, while
    \[
    \Ad(t)X_\alpha=\alpha(t)X_\alpha.
    \]
    Thus the dual covectors have the opposite weights.

    Now, we apply Lemma \ref{lem:dominant_bad_weights}. For every $n\in \N$, let $V_n$ be the irreducible representation of $G$ which has highest weight $\lambda_n$ and choose a highest-weight vector $0\neq v_n^+\in V_{\lambda_n}$.
    For $\ell\in V_{\lambda_n}^*$, define
    \[
    f_{n,\ell}(g)=\ell(\pi_{\lambda_n}(g)v_n^+),
    \qquad g\in G.
    \]
    Then, for every $H\in\mathfrak t_\C$, we get
    $Hf_{n,\ell}=\lambda_n(H)f_{n,\ell}$,
    and, for every $X\in\mathfrak b_{\mathfrak t}$, we have $Xf_{n,\ell}=0$.
    Indeed,
    \[
    (Yf_{n,\ell})(g)
    =
    \frac{d}{ds}\bigg|_{s=0}
    \ell(\pi_{\lambda_n}(g\exp(sY))v_n^+)
    =
    \ell(\pi_{\lambda_n}(g)d\pi_{\lambda_n}(Y)v_n^+),
    \]
    and the assertions follow from the defining properties of $v_n^+$.
    
    Fix $\ell_n\in V_{\lambda_n}^*$ so that $f_n \doteq f_{n,\ell_n}$ is nonzero, and normalize it by $\|f_n\|_{L^2(G)}=1$.
    Then
    \[
    \bar\partial_b f_n
    =
    \sum_{j=1}^r \lambda_n(L_j)f_n\,\tau^j.
    \]
    Set $u_n \doteq \bar\partial_b f_n$. If $u_n=0$ for infinitely many $n$, then the functions $f_n$ are nonzero CR
    functions belonging to pairwise distinct irreducible representations. Hence
    $H^0(G,\mathfrak v)$ is infinite-dimensional, and the theorem follows. We may
    therefore pass to a subsequence and assume that $u_n\neq0$ for every $n$.
    
    We claim that the series
    \[
        u \doteq \sum_{n=1}^{\infty}u_n
    \]
    converges in $\mathscr C^\infty(G,\underline \Lambda^1)$. Let $s\geq0$. Since the $f_n$'s are
    matrix coefficients of irreducible representations whose highest weights are
    $\lambda_n$, there exist constants $C_s>0$ and $M_s>0$, independent of $n$,
    such that
    \[
    \|f_n\|_{H^s(G)}\leq C_s(1+|\lambda_n|)^{M_s}.
    \]
    Therefore
    \[
    \|u_n\|_{H^s(G)}
    \leq
    C_s'\varepsilon_n(1+|\lambda_n|)^{M_s}
    \leq
    C_s'(1+|\lambda_n|)^{-2n+M_s}.
    \]
    After passing to a further subsequence if necessary, the last expression is
    summable in $n$. Hence $u \in \mathscr C^\infty(G, \underline \Lambda^{1})$. Since each
    $u_n=\bar\partial_b f_n$ is $\bar\partial_b$-closed, $u$ is
    $\bar\partial_b$-closed.
    
    We now prove that $u$ is not $\bar\partial_b$-exact. Suppose, by contradiction,
    that there exists $F \in \mathscr C^\infty(G)$ such that
    \[
    \bar\partial_bF=u.
    \]
    For $\chi\in X^*(\T)$, let $P_\chi$ denote the projection onto the $\chi$-component
    for the right $\T$-action on forms:
    \[
    P_\chi a
    =
    \int_\T \chi(t)^{-1}R_t^\ast a\,\mathrm{d}t.
    \]
    Since $\bar\partial_b$ commutes with $R_t^\ast$, it commutes with $P_\chi$.
    
    Each $f_n$ has coefficient right $\T$-weight $\lambda_n$, and $u_n$ has total
    right $\T$-weight $\lambda_n$, because the forms $\tau^j$ have weight $0$.
    The weights $\lambda_n$ are pairwise distinct. Therefore, applying
    $P_{\lambda_n}$ to the equation $\bar\partial_bF=u$, we obtain
    \[
    \bar\partial_b(P_{\lambda_n}F)=u_n=\bar\partial_b f_n.
    \]
    For functions, total right $\T$-weight is the same as coefficient right
    $\T$-weight, so $P_{\lambda_n}F$ is the usual coefficient right $\T$-weight
    $\lambda_n$ component of $F$.
    
    Set $h_n \doteq P_{\lambda_n}F-f_n$. Then $\bar\partial_b h_n=0$. Moreover, $h_n$ has coefficient right $\T$-weight $\lambda_n$. Hence, for
    every $j$, $L_jh_n=\lambda_n(L_j)h_n$.
    On the other hand, since $h_n$ is CR and $L_j\in\mathfrak m\subset
    \mathfrak v$, we also have $L_jh_n=0$.
    Because $u_n\neq0$, at least one of the numbers $\lambda_n(L_j)$ is nonzero.
    It follows then that  $h_n=0$. Consequently,
    \[
    P_{\lambda_n}F=f_n
    \]
    for every $n$.
    
    This contradicts the smoothness of $F$. Indeed, let $\Delta_G$ be the
    Laplace--Beltrami operator associated with a bi-invariant metric on $G$.
    Each $f_n$ belongs to an irreducible $G$-type whose eigenvalue
    $\Lambda_n$ tends to infinity. Since $\|f_n\|_{L^2}=1$, the equality
    $P_{\lambda_n}F=f_n$ implies $\|P_{\lambda_n}F\|_{L^2}=1$.
    But if $F \in \mathscr C^\infty(G)$, then for every $N\geq0$ the Fourier components of
    $F$ satisfy rapid decay with respect to $\Lambda_n$. In particular, $\Lambda_n^N\|P_{\lambda_n}F\|_{L^2}$
    must remain bounded as $n\to\infty$. This is impossible because
    $\Lambda_n \to \infty$. Therefore $u$ is not exact, thus defining a non-zero element of $H^1(G,\mathfrak v)$.

    Notice that we proved that $\dbarb : \mathscr C^\infty(G) \to \mathscr C^\infty(G, \underline \Lambda^1)$ does not have closed range, and consequently $H^1(G,\mathfrak v)$ is non-Hausdorff in the quotient topology, which prevents $H^1(G,\mathfrak v)$ from being finite-dimensional, and the result is established.

\bibliographystyle{alpha}
\bibliography{references.bib, other_references.bib}

\end{document}